\documentclass[11pt]{article}
\usepackage{amsfonts}
\usepackage{graphics}
\usepackage{indentfirst}
\usepackage{cite}
\usepackage{latexsym}
\usepackage{amsmath}
\usepackage{amssymb}
\usepackage[dvips]{epsfig}
\usepackage{amscd}
\usepackage{mathtools}
\allowdisplaybreaks
\newtheorem{theorem}{Theorem}[section]
\newtheorem{remark}{Remark}[section]

\newtheorem{lemma}[theorem]{Lemma}

\newtheorem{proposition}[theorem]{Proposition}

\newcommand{\n}{\rho}

\newcommand{\lm}{\lambda}

\newcommand{\ltwo}{_{L^2}^2}

\renewcommand{\div}{ {\rm div }  }
\def\on{\bar\rho}

\newcommand{\pa}{\partial}
\renewcommand{\r}{\mathbb{R}}

\newcommand{\ia}{\int_0^T}

\newcommand{\bt}{\begin{theorem}}
\newcommand{\bl}{\begin{lemma}}
\newcommand{\el}{\end{lemma}}
\newcommand{\et}{\end{theorem}}
\newcommand{\ga}{\gamma}

\newcommand{\curl}{{\rm curl} }

\newcommand{\de}{\delta}
\newcommand{\ve}{\varepsilon}

\newcommand{\ol}{\overline}

\newcommand{\bn}{\begin{eqnarray}}
\newcommand{\en}{\end{eqnarray}}
\newcommand{\bnn}{\begin{eqnarray*}}
\newcommand{\enn}{\end{eqnarray*}}

\newcommand{\bnnn}{\begin{eqnarray*}}
\newcommand{\ennn}{\end{eqnarray*}}

\newcommand{\ba}{\begin{aligned}}
\newcommand{\ea}{\end{aligned}}
\newcommand{\be}{\begin{equation}}
\newcommand{\ee}{\end{equation}}
\def\O{{\Omega }}

\def\norm[#1]#2{\|#2\|_{#1}}

\newcommand{\si}{\sigma}

\def\na{\nabla}
\def\on{\bar\n}

\def\QEDopen{{\setlength{\fboxsep}{0pt}\setlength{\fboxrule}{0.2pt}\fbox{\rule[0pt]{0pt}{1.3ex}\rule[0pt]{1.3ex}{0pt}}}} %???????

\def\QED{\QEDopen} % ??\QEDclosed?????
\def\endproof{\hspace*{\fill}~\QED\par\endtrivlist\unskip}% ?proof????????????????

\usepackage[marginal]{footmisc}
\renewcommand{\thefootnote}{}
\newcommand\blfootnote[1]{%
  \begingroup
  \renewcommand\thefootnote{}\footnote{#1}%
  \addtocounter{footnote}{-1}%
  \endgroup
}
\makeatletter      % '@' is now a normal "letter" for TeX
\@addtoreset{equation}{section}
\makeatother       % '@' is restored as a "non-letter" character for TeX

\title{Global Well-posedness of Classical Solutions to the Compressible Navier-Stokes-Poisson Equations with Slip Boundary Conditions in 3D Bounded Domains
}

\author{Yazhou C{\small HEN}, Bin H{\small UANG}%$^*$
, Xiaoding S{\small HI} \\
{\normalsize   College of Mathematics and Physics, }\\ {\normalsize  Beijing University of Chemical Technology, Beijing 100029, P. R. China} }

\date{ }

\begin{document}
\maketitle
%\blfootnote{*~Corresponding author. \\ { E-mail addresses: chenyz@mail.buct.edu.cn (Y. Chen), abinhuang@gmail.com (B. Huang), shixd@mail.buct.edu.cn (X. Shi)} }
\blfootnote{E-mail addresses: chenyz@mail.buct.edu.cn (Y. Chen), abinhuang@gmail.com (B. Huang), shixd@mail.buct.edu.cn (X. Shi)}
\begin{abstract}
We consider the initial-boundary-value problem of the isentropic compressible Navier-Stokes-Poisson equations subject to large and non-flat doping profile in 3D bounded domain with slip boundary condition and vacuum. The global well-posedness of classical solution is established with small initial energy but possibly large oscillations and vacuum. The steady state (except velocity) and the doping profile are allowed to be of large variation.
\end{abstract}

\textbf{Keywords:} Navier-Stokes-Poisson equations;  global classical solutions; slip boundary condition; large oscillations; non-flat doping profile.

\textbf{AMS Subject Classifications:} 35Q35, 35B40, 76N10

%MSC:35B65; 35M10; 35Q35;

\section{Introduction}
In this paper, we consider the compressible Navier-Stokes-Poisson (NSP) equations for the dynamics of charged particles of electrons (see \cite{MRS1990}) in a domain $\Omega\subset\r^{3}$, which can be written as
\begin{equation}\label{NSP}
\begin{cases}
\rho_t+ \mathop{\mathrm{div}}\nolimits(\rho u)=0,\\
(\rho u)_t+\mathop{\mathrm{div}}\nolimits(\rho u\otimes u)+\nabla P
=\mu \Delta u+(\mu+\lambda)\nabla \mathop{\mathrm{div}}\nolimits u +\rho\nabla\Phi,\\
\Delta \Phi=\rho-\tilde{\rho},
\end{cases}
\end{equation}
where $(x,t)\in\Omega\times (0,T]$, $t\geq 0$ is time, $x\in\r^{3}$ is the spatial coordinate. The unknown functions $\rho, u=(u^1,u^2,u^3), P=P(\rho)$ and $\Phi$ denote the electron density, the particle velocity, pressure and the electrostatic potential, respectively. Here we consider the isentropic flows with $\gamma$-law pressure $P(\rho)=a\rho^{\gamma}$, where $a>0$ and $\gamma >1$ are some physical parameters. The given function $\tilde{\rho}=\tilde{\rho}(x)>0$ is the doping profile, which describes the density of fixed, positively charged background ions. The constants $\mu$ and $\lambda$ are the shear viscosity and bulk coefficients respectively satisfying $\mu>0, 2\mu +3\lambda\geq 0.$
In addition, the system is solved subject to the given initial data
\begin{equation}\label{initial}
\displaystyle  \rho(x,0)=\rho_0(x), \quad \rho u(x,0)=\rho_0 u_0(x),\quad x\in \Omega,
\end{equation}
and boundary condition
\begin{align}
& u\cdot n=0,\,\,\,\curl u\times n=0, &\text{on} \,\,\,\partial\Omega, \label{navier-b}\\
& \nabla \Phi \cdot n=0,\,\,\,  &\text{on} \,\,\,\partial\Omega,\label{boundary}
\end{align}
where $n$ is the unit outer normal to $\partial \Omega$.
The boundary condition \eqref{navier-b} on the velocity is a special Navier-type slip boundary condition (see \cite{cl2019,xx2007}), in which there is a stagnant layer of fluid close to the wall allowing a fluid to slip, and the slip velocity is proportional to the shear stress.
This type of boundary condition was first introduced by Navier in \cite{Nclm1}, which was followed by great many applications, numerical studies and analysis for various fluid mechanical problems, see, for instance \cite{cf1988,Hoff2005,%Beirao2005,itt2003,,KZ1999
Zaja1998,xx2007,BZ1997} and the references therein.
For the electrostatic potential, we consider the Neumann boundary condition \eqref{boundary}, which describes that  the boundary of the domain is insulated (see\cite{GS2006}). It should be pointed out that
\begin{equation}
\displaystyle \int (\rho-\tilde{\rho}) dx=0.
\end{equation}
Furthermore, we observed that the solvability of the system \eqref{NSP}-\eqref{boundary} implies that the initial density $\rho_0$ and the doping profile $\tilde{\rho}$ must satisfy
\begin{equation}\label{d-d}
\displaystyle \int (\rho_0-\tilde{\rho}) dx=0.
\end{equation}
Moreover, when we consider the Poisson equation with the Neumann boundary condition in a bounded domain, we need the following uniqueness condition on $\Phi$,
\begin{equation}\label{phi-00}
\displaystyle \int\Phi dx=0.
\end{equation}
Next, we consider the steady state of the system \eqref{NSP} with the non-flat doping profile $\tilde{\rho}(x)$. Assume that $(\rho_s, u_s, \Phi_s)$ with $u_s \equiv 0$ is the stationary solution of \eqref{NSP}. Then it follows that
\begin{equation}\label{NSP-s}
\begin{cases}
\nabla P(\rho_s)=\rho_s\nabla\Phi_s\,\,\,  &\text{in} \,\,\,\Omega,\\
\Delta \Phi_s=\rho_s-\tilde{\rho}\,\,\,  &\text{in} \,\,\,\Omega,\\
\nabla\Phi_s \cdot n=0\,\,\,  &\text{on} \,\,\,\partial\Omega.
\end{cases}
\end{equation}
Let $Q(\rho)$ be the enthalhy function defined by
\begin{equation}\label{qrho}
 \displaystyle Q'(\rho)=P'(\rho)/\rho,%\frac{P'(\rho)}{\rho},
\end{equation}
which can change \eqref{NSP-s} to a quasilinear elliptic PDE and the classical fixed-point therem or a variational method can be used to prove the existence of stationary solution, without the smallness assumption on the oscillation of $\tilde{\rho}(x)$ (see \cite{GS2006,XJW2003}). We will record the existence and uniqueness of the solution to \eqref{NSP-s} in Lemma \ref{lem-s}.

The compressible Navier-Stokes-Poisson system \eqref{NSP} has been attracted a lot of attention and significant progress has been made in the analysis of the well-posedness and dynamic behavior to the solutions of the system.
We briefly review some results related to the global existence of strong (classical) solutions to Cauchy problem of the multi-dimensional compressible Navier-Stokes-Poisson system \eqref{NSP}-\eqref{initial}.
When the doping profile is flat, $i.e., \tilde{\rho}(x)\equiv \bar{\rho}>0$, the steady state of \eqref{NSP} is the trivial constant one $(\bar{\rho},0,0)$.
Global existence and the decay rates of the classical solution to its steady state were studied in \cite{HL2009,XL2010,LMZ2010,WwkW2010,Wang2012,Zheng2012,WW2015,LZ2012,BWY2017} in the Sobolev spaces or Besove spaces framework. This is extended to non-isentropic Navier-Stokes-Poisson systems, see \cite{TZ2013,ZL2012,ZLZ2011,TW2012,WW2012}.
While the doping profile is non-flat, global well-posedness of classical solutions and stability of the steady state were obtained in \cite{TWW2015,FL2018}, under analogous smallness conditions as the ones in \cite{LMZ2010,Wang2012}.
All of the above results need the smallness assumption on the oscillations between the initial data and the steady state. In particular, the initial density is near the non-vacuum steady density, which indicates that the density is uniformly away from the vacuum.
Recently, for Cauchy problem of the compressible Navier-Stokes-Poisson system \eqref{NSP}-\eqref{initial} subject to large and non-flat doping profile, global existence and uniqueness of strong solutions with large oscillations and vacuum was established in \cite{LXZ2020}, provided the initial data are of small energy and the steady state is strictly away from vacuum.

For the initial-boundary-value problem of the compressible Navier-Stokes-Poisson system \eqref{NSP}-\eqref{initial} with non-slip boundary condition for the velocity field, in the case that the doping profile is flat, the local and global existence of weak solution were obtained in \cite{Don2003,KS2008} and the local existence of unique strong solutions with vacuum was proved in \cite{TZ2010} with the initial data $\rho_0$ and $u_0$ satisfy a nature compatibility condition. Recently, for the case that the doping profile is non-flat, global existence of smooth solutions near the steady state for compressible Navier-Stokes-Poisson equations was established with the exponential stability in \cite{LZ2020}, with the smallness assumption on the oscillations between the initial data and the steady state but lager doping profile.
We also mentioned that the global existence of solutions to compressible Navier-Stokes-Poisson equations with the large initial data on a domain exterior to a ball was proved in \cite{LLZ2020} with the radial symmetry assumption.

However, there are no works about the global existence of the strong (classical) solution to the initial-boundary-value problem \eqref{NSP}-\eqref{boundary} for general bounded smooth domains $\Omega\in \mathbb{R}^3$ with initial density containing vacuum, at least to the best of our knowledge. Recently, for the barotropic compressible Navier-Stokes equations in $\Omega$ with slip boundary condition, the global classical solutions with large oscillations and vacuum to the initial-boundary-value problem was established in \cite{cl2019}, with some new estimates on boundary integrals related to the slip boundary condition.% compared with the work in \cite{HLX2012,lx2016} for the Cauchy problem of the compressible Navier-Stokes equations.

Based on the above research works, we study the global existence of the classical solutions with large oscillations and vacuum to the initial-boundary value problem \eqref{NSP}-\eqref{boundary} subject to large and non-flat doping profile in a bounded domain $\Omega\subset\r^3$.
Motivated by the works in \cite{cl2019,LXZ2020}, we would like to obtain the time-independent upper bound of the density and the time-dependent higher-norm estimates of $(\rho,u,\Phi)$, and extend the classical solution globally in time.
In our system, we need to conquer the difficulties arising from the coupling of the density $\rho$ with the electric field $\nabla\Phi$ and the slip boundary conditions \eqref{navier-b}. Firstly, compared with the previous results (see \cite{LXZ2020,HLX2012}) where they treated the Cauchy problem, it can not get the $L^p$-norm ($2\leq p\leq 6$) of $\nabla u$ by the standard elliptic estimate, due to the bounded domain with the slip boundary condition. To deal with this difficulty, we utilize $L^p$-theory for the div-curl system to control $\nabla u$ by means of $\div u$ and $\curl u$ (Lemma \ref{lem-vn}). %Moreover, we consider the Petrovsky type Lam\'{e}'s system \eqref{lame1} and obtain the estimates of $W^{k,q}$-norm to the slip boundary condition (see Lemma \ref{lem-lame}).
For the electric field $\nabla\Phi$, with the help of the  the classical regularity theory for the Neumann problem of elliptic
equation, we obtain the estimates \eqref{phi-hk} and \eqref{phi-dt}, which can be used to deal with the coupling between the density and the electric field.
%The slip boundary aslo causes additional difficulties in developing a priori estimates for the effective viscous flux $F$ and the vorticity $\omega$ (see \eqref{flux}, which play an important role to derive the time-independent upper bound of the density.
Secondly, owing to the coupling term $(\rho-\rho_s)\nabla\Phi_s$ and $\rho \nabla(\Phi-\Phi_s)$ in \eqref{NSP-2}, we can not obtain the time-independent estimates of $\|\rho-\rho_s\|_{L^2(0,T;L^2)}$ or $\|P-P_s\|_{L^2(0,T;L^2)}$. It is also difficult to get  time-independent estimates of $\|\nabla(\Phi-\Phi_s)\|_{L^2(0,T;L^2)}$ due to the ellipticity of the Poisson equation \eqref{phi2}. As a result, we can not close the time-independent estimates $A_1(T)$ (see \eqref{As1}) in a similar manner as that in \cite{HLX2012,cl2019}. In order to overcome this difficult point, we divide the time-independent estimates $A_1(T)$ into two time intervals $(0,\sigma(T))$ and $(\sigma(T),T)$. On the one hand, we can derive the short-time estimates \eqref{a1b} with the help of time-weighted estimates (see Lemma \ref{lem-a1}). On the other hand, we obtain the long-time estimates \eqref{a06} by the time-piecewise iterative argument, which was used in \cite{LXZ2020}.
Furthermore, we also obtain the time-dependent estimate $B[t_1,t_2]$, which can be bounded by the initial energy and the factor $t_2-t_1$ for any $1\leq t_1\leq t_2\leq T$ (see \eqref{ajf}). These estimates help us to derive the uniform (in time) upper bound for the density.
Additionally, we also need to pay more attention to control the boundary integrals during we derive the time-independent estimates of $A_1(T)$. It should be pointed out that the boundary condition $u\cdot n=0$ yields
\begin{align}\label{bdd1}
\displaystyle  u\cdot\nabla u\cdot n=-u\cdot\nabla n\cdot u,
\end{align}
which is the key to estimate the integrals on the boundary $\partial\Omega$ and we obtain the estimate of $\dot{u}$ and $\nabla \dot{u}$ (see Lemma \ref{lem-ud}). In order to estimate the high order derivatives of the solutions, we recall the similar Beale-Kato-Majda-type inequality (see Lemma \ref{lem-bkm}) with the respect to the slip boundary condition to prove the important estimates on the gradients of the density and velocity.

Before formulating our main result, we first explain the notation and conventions used throughout the paper.
For integer $k\geq 1$ and $1\leq q<+\infty$, We denote the standard Sobolev space by $W^{k,q}(\Omega)$ and $H^k(\Omega)\triangleq W^{k,2}(\Omega)$.
For some $\beta\in(0,1)$, the fractional Sobolev space $H^\beta(\Omega)$ is defined by
$$ H^\beta(\Omega)\triangleq\left\{u\in L^2(\Omega)~\text{:} \int_{\Omega\times\Omega}\frac{|u(x)-u(y)|^2}{|x-y|^{3+2\beta}}dxdy<+\infty\right\},\,\,\text{with the norm:}$$
$$\| u\|_{H^\beta(\Omega)}\triangleq \|u\|_{L^2(\Omega)}+\left(\int_{\Omega\times\Omega}\frac{|u(x)-u(y)|^2}{|x-y|^{3+2\beta}}dxdy\right)^\frac{1}{2}.$$
For simplicity, we denote $L^q(\Omega)$, $W^{k,q}(\Omega)$, $H^k(\Omega)$ and ${H^\beta(\Omega)}$ by $L^q$, $W^{k,q}$, $H^k$ and ${H^\beta}$ respectively, and set
$$\int fdx \triangleq \int_\Omega fdx,\quad \int_0^T\int fdx\triangleq\int_0^T\int_\Omega fdx. $$
For two $3\times 3$  matrices $A=\{a_{ij}\},\,\,B=\{b_{ij}\}$, the symbol $A\colon  B$ represents the trace of $AB^*$, where $B^*$ is the transpose of $B$, that is,
$$ A\colon  B\triangleq \text{tr} (AB^*)=\sum\limits_{i,j=1}^{3}a_{ij}b_{ij}.$$
Finally, for $v=(v^1,v^2,v^3)$, we denote $\nabla_iv\triangleq(\partial_iv^1,\partial_iv^2,\partial_iv^3)$ for $i=1,2,3,$ and the
material derivative of $v$   by  $\dot v\triangleq v_t+u\cdot\nabla v$.

Assume $\Omega$ is a simply connected bounded domain in $\r^3$ and its smooth boundary $\partial\Omega$ has a finite number of 2-dimensional connected components.
The initial total energy of \eqref{NSP} is defined as
\begin{align}\label{c0}
\displaystyle  C_0 \triangleq\int_{\Omega}\left(\frac{1}{2}\rho_0|u_0|^2 + G(\rho_0)+\frac{1}{2}|\nabla(\Phi_0-\Phi_s)|^2 \right)dx.
\end{align}
where $\Phi_0=\Phi(x,0)$ satisfies $\Delta \Phi_0=\rho_0-\tilde{\rho}$ in $\Omega$ and $\nabla \Phi_0 \cdot n=0$ on $\partial\Omega$, and $G(\rho)$ is the potential energy density given by
\begin{align}
\displaystyle  G(\rho)\triangleq\int_{\rho_s}^{\rho}\int_{\rho_s}^{\xi}\frac{P'(\zeta)}{\zeta} d\zeta d\xi=\rho\int_{\rho_s}^{\rho}\frac{P(\xi)-P_s}{\xi^{2}} d\xi.
\end{align}
where $P_s=P(\rho_s)$.

Now we can state our main result, Theorem \ref{th1}, concerning existence of global classical solutions to the problem  \eqref{NSP}-\eqref{boundary}.
\begin{theorem}\label{th1}
Let $(\rho_s,\Phi_s)$ be the stationary solutions of \eqref{NSP-s}. Assume that the smooth function $\tilde{\rho}(x)$ satisfy $0<\underline{\rho}\leq\tilde{\rho}(x)\leq \bar{\rho}$. For $q\in (3,6)$ and some given constants $M>0$, $\beta\in (\frac{1}{2},1]$, and $\hat{\rho}\geq \bar{\rho}+1$ , suppose that the initial data $(\rho_0, u_0)$ satisfy \eqref{navier-b}, \eqref{d-d} and
\begin{align}
& (\rho_0,P(\rho_0))\in H^2\cap W^{2,q},\quad u_0\in H^2, \label{dt1}\\
& 0\leq\rho_0\leq\hat{\rho},\quad \|u_0\|_{H^\beta}\leq M, \label{dt2}
\end{align}
and the compatibility condition
\begin{align}\label{dt3}
\displaystyle  -\mu\triangle u_0-(\mu+\lambda)\nabla \mathop{\mathrm{div}}\nolimits u_0 + \nabla P(\rho_0) = \rho_0^{1/2}g,
\end{align}
for some  $ g\in L^2.$
Then there exists a positive constant $\ve$ depending only on  $\mu$, $\lambda$,   $\ga$, $a$, $\on$, $\hat{\rho}$, $\beta$, $\Omega$ and  $M$ such that for the initial energy $C_0$ as in \eqref{c0} if
\begin{equation}
\displaystyle  C_0\leq\ve,
\end{equation}
the initial-boundary-value problem \eqref{NSP}-\eqref{boundary}, \eqref{phi-00} has a unique global classical solution $(\rho,u,\Phi)$ in $\Omega\times(0,\infty)$ satisfying
\begin{align}\label{esti-rho}
\displaystyle  0\le \rho(x,t)\le 2\hat{\rho},\quad  (x,t)\in \Omega\times(0,\infty),
\end{align}
\begin{equation}\label{esti-uh}
\begin{cases}
(\rho,P)\in C([0,\infty);H^2 \cap W^{2,q}),\\
\nabla u\in C([0,\infty);H^1 )\cap  L^\infty_{\rm loc}(0,\infty;H^2\cap W^{2,q}),\\
u_t\in L^{\infty}_{\rm loc}(0,\infty; H^2)\cap H^1_{\rm loc}(0,\infty; H^1),\\
\nabla\Phi \in C([0,\infty);H^2 \cap W^{2,q})\cap  L^\infty(0,\infty; H^3\cap W^{3,q}),\\
\nabla\Phi_t \in C([0,\infty);L^2)\cap  L^\infty(0,\infty; H^2).
\end{cases}
\end{equation}
\end{theorem}

\begin{remark}\label{rem:3} Compared with the results about global existence of the classical solutions mentioned above (see, for instance, \cite{LMZ2010,TWW2015,LZ2020}), our conclusion does not need the smallness assumption on the oscillations between the initial data and the steady state. Although it has small energy, its oscillations could be arbitrarily large. In particular, the initial vacuum states are allowed.
\end{remark}

\begin{remark}\label{rem:1} From \eqref{esti-uh}, Sobolev's inequality and the embedding
 $$L^2(\tau,T;H^1)\cap H^1(\tau,T;H^{-1})\hookrightarrow C([\tau,T];L^2),$$
the solution obtained in Theorem \eqref{th1} becomes a classical one away
from the initial time.
\end{remark}

\begin{remark}\label{rem:2}
When we consider the following general slip boundary for the velocity field
\begin{equation}
\displaystyle u\cdot n=0,\,\,\,\curl u\times n=-Au\, \, \, \text{on} \,\,\,\partial\Omega,
\end{equation}
and assume that the matrix $A$ is smooth and positive semi-definite, and even if the restriction on $A$ is relaxed to $A\in H^3$ and the negative eigenvalues of $A$ (if exist) are small enough, Theorem \ref{th1} will still hold. This can be achieved by a similar way as in \cite{cl2019}. Roughly speaking, we generalize the results of \cite{cl2019} to the compressible Navier-Stokes-Poisson equations.
\end{remark}

The rest of the paper is organized as follows. In Section \ref{se2},  the existence theorem of the steady-state solution and the local strong solution, and some key a priori estimates needed in later analysis are collected. Section \ref{se3} is  devoted to deriving the necessary a priori estimates on classical solutions which can guarantee the local classical solution to be a global classical one. In Section \ref{se5}, the proof of Theorem \ref{th1} will be completed. In Appendix \ref{appendix-a}, we list some elementary inequalities and important lemmas that we use intensively in the paper.

\section{Auxiliary lemma}\label{se2}
In this section, we recall the steady-state solution and the local strong (classical) solution of the system \eqref{NSP}-\eqref{boundary}. We also derive some key a priori estimates, which will be used frequently later.

First, similar to the proof of \cite{XJW2003,GS2006}, we have the following existence and uniqueness of the solution to \eqref{NSP-s}.
\begin{lemma}\label{lem-s}
Assume that the smooth function $\tilde{\rho}(x)$ satisfy $0<\underline{\rho}\leq\tilde{\rho}(x)\leq \bar{\rho}$. Then the problem \eqref{NSP-s} has a uniqueness classical solution $(\rho_s,\Phi_s)$. Moreover
\begin{equation}\label{rho-s}
\displaystyle  \underline{\rho}\leq\rho_s(x)\leq \bar{\rho},
\end{equation}
\begin{equation}\label{h3-s}
\displaystyle \|\nabla \rho_s\|_{H^3}+ \|\nabla \Phi_s\|_{H^4}\leq C,
\end{equation}
where $C$ depends only on $a, \gamma, \tilde{\rho}(x)$ and $\Omega$.
\end{lemma}

The following local existence theorem of classical solution of \eqref{NSP}-\eqref{boundary} can be proved in a similar manner as that in \cite{TZ2010}, base on the standard contraction mapping principle.
\begin{lemma}\label{lem-local}
Assume that the initial date $(\rho_0,u_0)$ satisfy the conditions \eqref{dt1} and \eqref{dt3}. Then there exist a positive time $T_0>0$ and a unique classical solution $(\rho,u,\Phi)$ of the system \eqref{NSP}-\eqref{boundary} in $\r^3 \times (0,T_0]$, satisfying that $\rho \geq 0$, and that for $\tau \in (0,T_0)$,
\begin{equation}\label{esti-uh-local}
\begin{cases}
(\rho,P)\in C([0,T_0);H^2 \cap W^{2,q}),\\
\nabla u\in C([0,T_0);H^1 )\cap  L^\infty(\tau,T_0;H^2\cap W^{2,q}),\\
u_t\in L^\infty(\tau,T_0; H^2)\cap H^1(\tau,T_0; H^1),\\
\nabla \Phi \in C([0,T_0);H^2 \cap W^{2,q})\cap  L^\infty(\tau,T_0; H^3\cap W^{3,q}),\\
\nabla \Phi_t\in C([0,T_0);L^2)\cap L^\infty(\tau,T_0; H^2).
\end{cases}
\end{equation}
\end{lemma}

Next, we denote the effective viscous flux $F$ and the vorticity $\omega$ by
\begin{align}\label{flux}
F\triangleq(\lambda+2\mu)\text{div}u-(P-P_s),\quad \omega \triangleq \nabla\times u,
\end{align}
which plays an important role in our following analysis, similarly to that for the compressible Navier-Stokes equations (see \cite{Hoff1995,lx2016,HLX2012}). For $F$, $\omega$ and  $\nabla u$, we give the following conclusion, which is a key to a priori estimates.
\begin{lemma}\label{lem-f-td}
Let $(\rho,u,\Phi)$ be a smooth solution of \eqref{NSP}-\eqref{boundary}. Then for any $p\in[2,6],$ there exists a positive constant $C$ depending only on $p$, $\mu$, $\lambda$ and $\Omega$ such that
\begin{align}
&\|\nabla F\|_{L^p}\leq C(\|\rho\dot{u}\|_{L^p}+\|\rho\nabla(\Phi-\Phi_s)\|_{L^p}+\|(\rho-\rho_s)\nabla\Phi_s\|_{L^p}),\label{tdf1}\\
& \|\nabla\omega\|_{L^p}\leq C(\|\rho\dot{u}\|_{L^p}+\|\rho\nabla(\Phi-\Phi_s)\|_{L^p}+\|(\rho-\rho_s)\nabla\Phi_s\|_{L^p}+\|\nabla u\|_{L^2}),\label{tdxd-u1} \\
&\|F\|_{L^p}\leq C(\|\rho\dot{u}\|_{L^2}\!+\!\|\rho\nabla(\Phi\!-\!\Phi_s)\|_{L^2}\!+\!\|(\rho\!-\!\rho_s)\nabla\Phi_s\|_{L^2})^\frac{3p-6}{2p}(\|\nabla u\|_{L^2}\!+\!\|P\!-\!P_s\|_{L^2})^\frac{6-p}{2p}\nonumber \\
& +C(\|\nabla u\|_{L^2}+\|P-P_s\|_{L^2}),\label{f-lp}\\
&\|\omega\|_{L^p}\leq  C(\|\rho\dot{u}\|_{L^2}\!+\!\|\rho\nabla\!(\Phi\!-\!\Phi_s)\|_{L^2}\!+\!\|(\rho\!-\!\rho_s)\nabla\!\Phi_s\|_{L^2})^\frac{3p-6}{2p}\|\nabla\! u\|_{L^2}^\frac{6-p}{2p}\!+\!C\|\nabla\! u\|_{L^2},\label{xdu1}
\end{align}
Moreover,
\begin{align}
\|\nabla u\|_{L^p}\leq & C(\|\rho\dot{u}\|_{L^2}\!+\!\|\rho\nabla(\Phi\!-\!\Phi_s)\|_{L^2}\!+\!\|(\rho\!-\!\rho_s)\nabla\Phi_s\|_{L^2})^\frac{3p-6}{2p}(\|\nabla u\|_{L^2}\!+\!\|P\!-\!P_s\|_{L^2})^\frac{6-p}{2p}\nonumber \\
&+C(\|\nabla u\|_{L^2}+\|P\!-\!P_s\|_{L^p}).\label{tdu2}
\end{align}
\end{lemma}
\begin{proof} By \eqref{NSP}$_2$, one can find that the viscous flux $F$ satisfies
\begin{equation}
\begin{cases}
\Delta F=\div(\rho\dot{u}-\rho\nabla(\Phi-\Phi_s)-(\rho-\rho_s)\nabla\Phi_s)~~ &in\,\,\Omega,\\ \frac{\partial F}{\partial n}=\rho\dot{u}\cdot n\,\, &on\,\, \partial\Omega.
\end{cases}
\end{equation}
It follows from Lemma 4.27 in \cite{NS2004}, for $1<q<+\infty$, that
\begin{align}\label{tdf2}
\displaystyle  \|\nabla F\|_{L^q}\leq C(\|\rho\dot{u}\|_{L^q}+\|\rho\nabla(\Phi-\Phi_s)\|_{L^q}+\|(\rho-\rho_s)\nabla\Phi_s\|_{L^q}),
\end{align}
so that \eqref{tdf1} holds.
%Moreover, for any integer $k\geq 0$,
%\begin{align}\label{tdfk} \|\nabla F\|_{W^{k+1,q}}\leq & C(\|\rho\dot{u}\|_{L^q}+\|\rho\nabla(\Phi-\Phi_s)\|_{L^q}+\|(\rho-\rho_s)\nabla\Phi_s\|_{L^q}\nonumber\\& +\|\nabla(\rho\dot{u})\|_{W^{k,q}}+\|\nabla(\rho\nabla(\Phi-\Phi_s)+(\rho-\rho_s)\nabla\Phi_s)\|_{W^{k,q}}).\end{align}
On the other hand, one can rewrite $\eqref{NSP}_2 $ as
\begin{equation}
\displaystyle  \mu\nabla\times\omega=\nabla F-\rho\dot{u}+\rho\nabla(\Phi-\Phi_s)+(\rho-\rho_s)\nabla\Phi_s.
\end{equation}
Notice that $\omega\times n=0$ on $\partial\Omega$ and $\mathop{\mathrm{div}}\nolimits \omega=0$, by Lemma \ref{lem-curl}, we get
\begin{align}
\displaystyle  \|\nabla\omega\|_{L^q}&\leq C(\|\nabla\times\omega\|_{L^q}+\|\omega\|_{L^q}) \nonumber \\
 &\leq C(\|\rho\dot{u}\|_{L^q}+\|\rho\nabla(\Phi-\Phi_s)\|_{L^q}+\|(\rho-\rho_s)\nabla\Phi_s\|_{L^q}+\|\omega\|_{L^q}),\label{tdxd-u2}
\end{align}
%and for any integer $k\geq0$,
%\begin{align}\label{tdxdk} &\quad\|\nabla\omega\|_{W^{k+1,q}}\leq C(\|\nabla\times\omega\|_{W^{k+1,q}}+\|\omega\|_{L^q})\nonumber \\&\leq C(\|\rho\dot{u}\|_{L^q}+\|\rho\nabla(\Phi-\Phi_s)\|_{L^q}+\|(\rho-\rho_s)\nabla\Phi_s\|_{L^q}+\|\nabla(\rho\dot{u})\|_{W^{k,q}}\nonumber \\&\qquad+\|\nabla(\rho\nabla(\Phi-\Phi_s)+(\rho-\rho_s)\nabla\Phi_s)\|_{W^{k,q}}+\|\omega\|_{L^q}),\end{align}
where we have taken advantage of \eqref{tdf2}. % and \eqref{tdfk}.
By Sobolev's inequality and \eqref{tdxd-u2}, for $p\in[2,6]$,
\begin{align}
\|\nabla\omega\|_{L^p}%&\le C(\|\rho\dot{u}\|_{L^p}+\|\rho\nabla(\Phi-\Phi_s)\|_{L^p}+\|(\rho-\rho_s)\nabla\Phi_s\|_{L^p}+\|\omega\|_{L^p} ) \nonumber \\
&\le C(\|\rho\dot{u}\|_{L^p}+\|\rho\nabla(\Phi-\Phi_s)\|_{L^p}+\|(\rho-\rho_s)\nabla\Phi_s\|_{L^p}+\|\nabla\omega\|_{L^2}+\|\omega\|_{L^2}) \nonumber \\
%&\le C(\|\rho\dot{u}\|_{L^p}+\|\rho\nabla(\Phi-\Phi_s)\|_{L^p}+\|(\rho-\rho_s)\nabla\Phi_s\|_{L^p}+\|\rho\dot{u}\|_{L^2}\nonumber \\ &\quad+\|\rho\nabla(\Phi-\Phi_s)\|_{L^2}+\|(\rho-\rho_s)\nabla\Phi_s\|_{L^2}+\|\omega\|_{L^2}) \nonumber \\
&\le C(\|\rho\dot{u}\|_{L^p}+\|\rho\nabla(\Phi-\Phi_s)\|_{L^p}+\|(\rho-\rho_s)\nabla\Phi_s\|_{L^p}+\|\nabla u\|_{L^2}),
\end{align}
which implies \eqref{tdxd-u1}.

Furthermore, one can deduce from \eqref{g1} and \eqref{tdf1} that for $p\in[2,6]$,
\begin{align}\label{f-lp1}
& \|F\|_{L^p}\leq C\|F\|_{L^2}^{\frac{6-p}{2p}}\|\nabla F\|_{L^2}^{\frac{3p-6}{2p}}+C\|F\|_{L^2}\nonumber \\
\leq & C(\|\rho\dot{u}\|_{L^2}+\|\rho\nabla(\Phi-\Phi_s)\|_{L^2}+\|(\rho-\rho_s)\nabla\Phi_s\|_{L^2})^{\frac{3p-6}{2p}}(\|\nabla u\|_{L^2}+\|P-P_s\|_{L^2})^{\frac{6-p}{2p}}\nonumber\\
&+C(\|\nabla u\|_{L^2}+\|P-P_s\|_{L^2}),
\end{align}
similarly, by \eqref{g1} and \eqref{tdxd-u1},
\begin{align}\label{xdu-lp1}
&\|\omega\|_{L^p}\leq C\|\omega\|_{L^2}^{\frac{6-p}{2p}}\|\nabla \omega\|_{L^2}^{\frac{3p-6}{2p}}+C\|\omega\|_{L^2}\nonumber \\
%\leq& C(\|\rho\dot{u}\|_{L^2}+\|\rho\nabla(\Phi-\Phi_s)\|_{L^2}+\|(\rho-\rho_s)\nabla\Phi_s\|_{L^2}+\|\nabla u\|_{L^2})^{\frac{3p-6}{2p}}\|\nabla u\|_{L^2}^{\frac{6-p}{2p}}+C\|\nabla u\|_{L^2} \nonumber \\
\leq& C(\|\rho\dot{u}\|_{L^2}+\|\rho\nabla(\Phi-\Phi_s)\|_{L^2}+\|(\rho-\rho_s)\nabla\Phi_s\|_{L^2})^{\frac{3p-6}{2p}}\|\nabla u\|_{L^2}^{\frac{6-p}{2p}}+C\|\nabla u\|_{L^2},
\end{align}
and so \eqref{f-lp}-\eqref{xdu1} are established.
By virtue of \eqref{tdu1}, \eqref{f-lp} and \eqref{xdu1}, it implies that \eqref{tdu2} holds. This completes the proof.
\end{proof}

\begin{remark}
 we can get the estimates of $\|\nabla^{2}u\|_{L^p}$ and $\|\nabla^{3}u\|_{L^p}$ for $p\in[2,6]$ by Lemma \ref{lem-vn}, which will be devoted to give higher order estimates in Section \ref{se3}. In fact, by Lemma \ref{lem-vn} and \eqref{tdf2} %and \eqref{tdfk}
 , for $p\in[2,6]$,
\begin{align}\label{2tdu}
&\quad\|\nabla^{2}u\|_{L^p}\leq C(\|\div u\|_{W^{1,p}}+\|\omega\|_{W^{1,p}})\nonumber \\
%&\leq C(\|F\|_{W^{1,p}}+\|P-P_s\|_{W^{1,p}}+\|\omega\|_{W^{1,p}})\nonumber \\
&\leq C(\|\rho\dot{u}\|_{L^p}\!\!+\!\|\rho\nabla\!(\Phi\!\!-\!\Phi_s)\|_{L^p}\!\!+\!\|(\rho\!-\!\!\rho_s)\nabla\!\Phi_s\|_{L^p}\!\!+\!\|\nabla\!(P\!\!-\!\!P_s)\|_{L^p}\!\!+\!\|P\!\!-\!P_s\|_{L^p}\!\!+\!\!\|\nabla\! u\|_{L^2}),
\end{align}
and
\begin{align}\label{3tdu}
&\quad\|\nabla^{3}u\|_{L^p}\leq C(\|\div u\|_{W^{2,p}}+\|\omega\|_{W^{2,p}})\nonumber \\
%&\leq C(\|\nabla^{2} F\|_{L^p}+\|\nabla^{2} P\|_{L^p}+\|\nabla^{2}\omega\|_{L^p}+\|\div u\|_{W^{1,p}}+\|\omega\|_{W^{1,p}})\nonumber \\
&\leq C(\|\nabla(\rho\dot{u})\|_{L^p}\!+\!\|\nabla\!(\rho\nabla\!(\Phi\!-\!\Phi_s))\|_{L^p}\!+\!\|\nabla\!((\rho\!-\!\rho_s)\nabla\!\Phi_s)\|_{L^p}\!+\!\|\nabla^2(P\!-\!P_s)\|_{L^p}\!+\!\|\rho\dot{u}\|_{L^p}\nonumber \\
&\quad +\!\|\rho\nabla\!(\Phi\!-\!\Phi_s)\|_{L^p}\!+\!\|(\rho\!-\!\rho_s)\nabla\!\Phi_s\|_{L^p}\!+\!\|\nabla\! (P\!\!-\!P_s)\|_{L^p}\!+\!\|P\!\!-\!P_s\|_{L^p}\!+\!\|\nabla\! u\|_{L^2}).
\end{align}
\end{remark}

To this end, we derive the estimates of $\nabla(\Phi-\Phi_s)$ and $\nabla(\Phi-\Phi_s)_t$.
\begin{lemma}\label{lem-phi}
Let $(\rho,u,\Phi)$ be a smooth solution of \eqref{NSP}-\eqref{boundary}. Then for any integer any integer $k\geq 0$, $q>1$, there exists a positive constant $C$ depending only on $k$, $q$ and $\Omega$ such that	
\begin{align}
& \|\nabla(\Phi-\Phi_s)\|_{W^{k+1,q}}\leq C \|\rho-\rho_s\|_{W^{k,q}},\label{phi-hk}\\
& \|\nabla(\Phi-\Phi_s)_t\|_{W^{k+1,q}}\leq C(\|\rho u\|_{L^q}+\|\rho u\|_{W^{k,q}}).\label{phi-dtk}
\end{align}
\end{lemma}
\begin{proof}
From \eqref{NSP}$_3$, \eqref{boundary} and \eqref{NSP-s}, we have
\begin{equation}
\begin{cases}\label{phi2}
\Delta (\Phi\!-\!\Phi_s)=\rho-\rho_s\,\,\,&\,in \,\,\,\Omega,	\\
\frac{\partial(\Phi-\Phi_s)}{\partial n}=0\,\,\,&\,on \,\,\,\partial\Omega,
\end{cases}
\displaystyle
\end{equation}
which, in view of the classical regularity theory for the Neumann problem of elliptic equation \cite{BB1974}, leads to \eqref{phi-hk}.
From \eqref{NSP}$_{1}$ and \eqref{phi2}, we obtain
\begin{equation}
\begin{cases}
\displaystyle \Delta (\Phi-\Phi_s)_t=-\div(\rho u)\,\,\, &in\,\,\Omega,\\
\frac{\partial (\Phi-\Phi_s)_t}{\partial n}=0\, \,\, &on\,\,\partial\Omega.
\end{cases}
\end{equation}
Using the same method as Lemma \ref{lem-f-td}, we obtain, for $1<q<+\infty$, that
%It follows from Lemma 4.27 in \cite{NS2004}, for $1<q<+\infty$, that
\begin{align}\label{phi-dt}
\displaystyle  \|\nabla (\Phi-\Phi_s)_t\|_{L^q}\leq C\|\rho u\|_{L^q},
\end{align}
Moreover, for any integer $k\geq 0$,
\begin{equation}
\displaystyle \|\nabla(\Phi-\Phi_s)_t\|_{W^{k+1,q}}\leq C(\|\rho u\|_{L^q}+\|\rho u\|_{W^{k,q}}).\label{phi-dtk1}
\end{equation}
The proof of Lemma \eqref{lem-phi} is completed.
\end{proof}

\begin{remark}\label{rem:phi}
It follows from \eqref{phi-dt}, if $\rho\leq 2\hat{\rho}$, that
\begin{align}\label{phi-dt2}
\displaystyle  \|\nabla (\Phi-\Phi_s)_t\|_{L^2}\leq C\|\rho u\|_{L^2}\leq C \|\nabla u\|_{L^2},
\end{align}
where $C$ depends on $\hat{\rho}$ and $\Omega$.
\end{remark}

\section{\label{se3} A priori estimates}

In this section, we will establish some necessary a priori bounds for smooth solutions to the problem \eqref{NSP}-\eqref{boundary} to extend the local classical solutions guaranteed by Lemma \ref{lem-local}.

Let $T>0$ be a fixed time and $(\rho,u,\Phi)$ be a smooth solution to \eqref{NSP}-\eqref{boundary} on $\Omega \times (0,T]$  with smooth initial data $(\rho_0,u_0)$ satisfying $u_0\in H^\beta$ for some $\beta\in(\frac{1}{2},1]$ and $0\leq\rho_0\leq 2\hat{\rho}$.
Set $\si=\si(t)\triangleq\min\{1,t \},$ we define
\begin{align}
& A_1(T) \triangleq \sup_{   0\le t\le T  }\big(\sigma\|\nabla u\|_{L^2}^2+\sigma^3\|\sqrt{\rho}\dot{u}\|_{L^2}^2\big),\label{As1}\\
& A_2(T) \triangleq\sup_{  0\le t\le T   }\int\rho|u|^3dx,\label{As3} \\
& B[t_1,t_2] \triangleq \int_{t_1}^{t_2}\big(\sigma\|\sqrt{\rho}\dot{u}\|_{L^2}^2+\sigma^3\|\nabla\dot{u}\|_{L^2}^2\big)dt,\label{As2}
\end{align}
where $0\leq t_1<t_2\leq T$ and $\dot{v}=v_t+u \cdot \nabla v$ is the material derivative.

Now we will give the following key a priori estimates in this section, which guarantees the existence of a global classical solution of \eqref{NSP}--\eqref{boundary}.
\begin{proposition}\label{pr1}
Under the conditions of Theorem \ref{th1}, for $\delta_0\triangleq\frac{2\beta-1}{4\beta}\in(0,\frac{1}{4}]$, there exists a  positive constant  $\ve$ depending on $\mu$, $\lambda$,   $a$, $\ga$, $\on$, $\hat{\rho}$, $\beta$, $\Omega$ and $M$  such that if $(\rho,u,\Phi)$  is a smooth solution of \eqref{NSP}-\eqref{boundary}  on $\Omega\times (0,T] $ satisfying
\begin{equation}\label{key1}
\sup\limits_{\Omega\times [0,T]}\rho\le 2\hat{\rho},\quad
A_1(T) \le 2C_0^{1/2},\quad  A_2(\sigma(T))\leq 2C_0^{\delta_0},
\end{equation}
then the following estimates hold
\begin{equation}\label{key2}
\sup\limits_{\Omega\times [0,T]}\rho\le 7\hat{\rho}/4,\quad
A_1(T) \le C_0^{1/2},\quad A_2(\sigma(T))\leq C_0^{\delta_0},
\end{equation}
provided $C_0\le \ve.$
\end{proposition}
\begin{proof} Proposition \ref{pr1} is a consequence of the following Lemmas \ref{lem-a3}-\ref{lem-brho}.
\end{proof}

In the following, we will use the convention that $C$ denotes a generic positive constant depending on $\mu , \lambda,  \ga ,  a , \rho_s, \hat{\rho},  \beta, \Omega$  and $M$ and use $C(\alpha)$ to emphasize that $C$ depends on $\alpha$. We begin with the following standard energy estimate for $(\rho,u,\Phi)$.
\begin{lemma}\label{lem-basic}
 Let $(\rho,u,\Phi)$ be a smooth solution of
 \eqref{NSP}--\eqref{boundary} on $\O \times (0,T]$ satisfying $\rho\leq 2\hat{\rho}$. Then there is a positive constant $C $ depending only  on $\mu,$  $\lambda,$   $a$, $\gamma$, $\hat{\rho}$ and $\Omega$  such that
\begin{align}
\displaystyle  &\sup_{0\le t\le T}
\left(\frac{1}{2}\|\sqrt{\rho}u\|_{L^2}^2+\|G(\rho)\|_{L^1}+\frac{1}{2}\|\nabla(\Phi-\Phi_s)\|_{L^2}^2\right) \nonumber \\
 &+ \int_0^{T}(\lambda+2\mu)\|\div u\|_{L^2}^2 +\mu\|\omega\|_{L^2}^2)dt \le C_0,\label{basic1}
\end{align}
and
\begin{align}\label{basic2}
\displaystyle \sup_{0\le t\le T}\|\rho-\rho_s\|_{L^2}^2+\int_{0}^{T}\|\nabla u\|_{L^{2}}^{2}dt\leq CC_{0}.
\end{align}
\end{lemma}
\begin{proof}
Note that $-\Delta u=-\nabla\div u+\nabla\times\omega$ and \eqref{NSP-s}, we
rewrite $\eqref{NSP}_2$ as
\begin{equation}\label{NSP-2}
\rho u+\rho u \cdot \nabla u- (\lambda + 2\mu)\nabla\div u+\mu\nabla\times\omega\! + \nabla(P\!-\! P_s)=\rho\nabla(\Phi\!-\!\Phi_s)+(\rho\!-\!\rho_s)\nabla\Phi_s.
\end{equation}
Multiplying \eqref{NSP-2}$_2$ by $u$ and \eqref{NSP}$_1$ by $G'(\rho)$ respectively, integrating over $\Omega$, summing them up, by \eqref{navier-b}-\eqref{boundary} and \eqref{phi2},  we have
\begin{align}\label{m2}
\displaystyle  &\left(\int\Big(G(\rho)+\frac{1}{2}\rho |u|^{2}%+\frac{1}{2}|\nabla(\Phi-\Phi_s)|^2
\Big)dx\right)_t + (\lambda + 2\mu)\int(\div u)^{2}dx + \mu\int|\omega|^{2}dx \nonumber \\
&=\int \mathop{\mathrm{div}}\nolimits(\rho u)G'(\rho)dx\!+\!\int u \cdot \nabla (P\!-\!P_s)dx
\!+\!\int(\rho\!-\!\rho_s)u \cdot \nabla \Phi_s dx\!+\!\int \rho u \cdot \nabla (\Phi\!-\!\Phi_s)dx \nonumber \\
&=\int \rho u \cdot \nabla Q(\rho)+\int u \cdot \nabla Pdx-\left(\int\frac{1}{2}|\nabla(\Phi-\Phi_s)|^2dx\right)_t \nonumber \\
&=-\left(\int\frac{1}{2}|\nabla(\Phi-\Phi_s)|^2dx\right)_t,
\end{align}
where we have used the fact $G'(\rho)=Q(\rho)-Q(\rho_s)$ and $Q'(\rho)=P'(\rho)/\rho$.
Then integrating \eqref{m2} over $(0,T)$ yields \eqref{basic1}.
Finally, it is easy to check that there exists a positive constant $C,$ depending only on $a, \gamma$ and $\hat{\rho}$, such that
$$C^{-1}(\rho-\rho_s)^{2}\leq G(\rho)\leq C(\rho-\rho_s)^{2},$$
which together with \eqref{tdu1} and \eqref{basic1} gives \eqref{basic2}.
The proof of Lemma \ref{lem-basic} is completed.
\end{proof}

Next we need the estimates on the material derivative of $u$. Since $u\cdot n=0$ on $\partial\Omega$, it follows that
% \begin{align}\label{bdd2}
% \displaystyle u\cdot\nabla u\cdot n=-u\cdot\nabla n\cdot u,
% \end{align}
% which implies
\begin{align}\label{bdd3}
\displaystyle (\dot{u}+(u\cdot\nabla n)\times u^{\perp})\cdot n=0 \mbox{ on } \partial \Omega,
\end{align}
where $u^{\perp}\triangleq -u\times n$ on $\partial\Omega$. In view of this observation, we review the following Poincare-type inequality of $\dot u$ (see \cite{cl2019}, Lemma 3.2).
\begin{lemma}\label{lem-ud}
If $(\rho,u,\Phi)$ is a smooth solution of \eqref{NSP} with slip condition \eqref{navier-b}-\eqref{boundary}, then there exists a positive constant $C$ depending only on $\Omega$ such that
\begin{align}
&\|\dot{u}\|_{L^6}\le C(\|\nabla\dot{u}\|_{L^2}+\|\nabla u\|_{L^2}^2),\label{udot}\\
&\|\nabla\dot{u}\|_{L^2}\le C(\|\div \dot{u}\|_{L^2}+\|\curl \dot{u}\|_{L^2}+\|\nabla u\|_{L^4}^2).\label{tdudot}
\end{align}
\end{lemma}

Then, we give the estimate of $A_1(T)$ and $B[0,T]$.
\begin{lemma}\label{lem-a0}
 Let $(\rho,u,\Phi)$ be a smooth solution of
 \eqref{NSP}-\eqref{boundary} satisfying \eqref{key1}.
  Then there is a positive constant
  $C$ depending only  on $\mu$, $\lambda$, $a$, $\gamma$, $\rho_s$, $\hat{\rho}$ and $\Omega$  such that
 \begin{align}\label{a00}
 \displaystyle  A_1(T)+B[0,T]
    \le   C C_0+ C\int_0^{T}\big(\sigma\|\nabla u\|_{L^3}^3+\sigma^3 \|\nabla u\|_{L^4}^4\big)dt.
\end{align}
%Moreover, for positive integer $k\geq 1$,
%\begin{equation}\label{a03}\sup_{k\le t\le k+1 }\|\nabla u\|_{L^2}^2 + \int_k^{k+1} \|\sqrt{\rho}\dot{u}\|_{L^2}^2 dt \leq \|\nabla u\|_{L^2}^2(k) +CC_0+C\int_k^{k+1} \|\nabla u\|_{L^3}^3 dt\end{equation}
%\begin{align}\label{a04} &\sup_{k\le t\le k+1 }\|\sqrt{\rho}\dot{u}\|_{L^2}^2 + \int_k^{k+1} \|\nabla\dot{u}\|_{L^2}^2 dt \nonumber \\  \leq & \|\sqrt{\rho}\dot{u}\|_{L^2}^2(k) +CC_0+C\int_k^{k+1} \|\sqrt{\rho}\dot{u}\|_{L^2}^2 dt  +C\int_k^{k+1} \|\nabla u\|_{L^4}^4 dt\end{align}
\end{lemma}
\begin{proof}
Let $m\ge 0$ be a real number which will be determined later.
Multiplying $\eqref{NSP}_2 $ by $\sigma^m \dot{u}$   and then integrating the resulting equality over
$\Omega$ lead  to
\begin{align}\label{I0}
\int \sigma^m \rho|\dot{u}|^2dx &
= -\int\sigma^m \dot{u}\cdot\nabla (P-P_s)dx + (\lambda+2\mu)\int\sigma^m \nabla\div u\cdot\dot{u}dx \nonumber\\
&\quad - \mu\int\sigma^m \nabla\times\omega\cdot\dot{u}dx
+\int\sigma^m \rho\dot{u}\nabla(\Phi-\Phi_s)dx \nonumber\\
&\quad+\int\sigma^m(\rho-\rho_s)\dot{u}\nabla\Phi_sdx \triangleq \sum_{i=1}^5I_i.
\end{align}
We have to estimate $I_i (i=1,\cdots,5)$ one by one.
Firstly, By $\eqref{NSP}_1$, one can check that
\begin{align}\label{Pu1}
\displaystyle  (P-P_s)_t+u \cdot \nabla (P-P_s)+\gamma P\div u+u \cdot \nabla P_s=0,
\end{align}
A direct calculation by applying  $\eqref{Pu1}$ gives
\begin{align}\label{I10}
I_1 %=& -\int\sigma^m u_{t}\cdot\nabla(P-P_s)dx - \int\sigma^m u\cdot\nabla u\cdot\nabla (P-P_s)dx \nonumber\\      = & \left(\int\sigma^m(P-P_s)\,\div u\, dx\right)_{t} - m\sigma^{m-1}\sigma'\int(P-P_s)\,\div u\, dx \nonumber\\ &- \int \sigma^{m}\div u\, (P-P_s)_{t}dx - \int\sigma^{m}u\cdot\nabla u\cdot\nabla(P-P_s)dx\nonumber\\
= & \left(\int\sigma^m(P-P_s)\,\div u\, dx\right)_{t} - m\sigma^{m-1}\sigma'\int(P-P_s)\,\div u\,dx \nonumber\\
&- \int\sigma^{m} u \cdot \nabla P_s \div u\,dx + \int\sigma^{m}(P-P_s)\nabla u:\nabla u dx \nonumber\\
&+ \int\sigma^{m}(\gamma P- P+P_s)(\div u)^{2}dx - \int_{\partial\Omega}\sigma^{m}(P-P_s)u\cdot\nabla u\cdot n ds\nonumber\\
\leq & \left(\int\sigma^m(P-P_s)\,\div u\, dx\right)_{t} + C\sigma^{m} \|\nabla u \|_{L^2}^2+ Cm\sigma^{m-1}\sigma'\|\rho-\rho_s\|_{L^2}^2\nonumber\\
& +\int_{\partial\Omega}\sigma^{m}(P-P_s)u\cdot\nabla n\cdot u ds.
\end{align}
where in last term on the right-hand side of \eqref{I10} we have used
the fact \eqref{bdd1} on $\partial \Omega$ and  we conclude that
\begin{align}\label{bdt1}
\displaystyle \int_{\partial\Omega}\sigma^{m}(P-P_s)u\cdot\nabla n\cdot uds
&\leq C\int_{\partial\Omega}\sigma^{m}|u|^{2}ds \leq C\sigma^{m}\|\nabla u\|_{L^{2}}^{2}.
\end{align}
Hence,
\begin{align}\label{I1}
\displaystyle I_1 \leq \left(\int\sigma^m(P-P_s)\,\div u\, dx\right)_{t} + C\sigma^{m} \|\nabla u \|_{L^2}^2+ Cm\sigma^{m-1}\sigma'\|\rho-\rho_s\|_{L^2}^2.
\end{align}
Similarly, %by \eqref{bdd1}, it indicates that
\begin{align}\label{I20}
I_2 %& =  (\lambda+2\mu)\int\sigma^m \nabla\div u\cdot\dot{u}dx \nonumber\\
%& = (\lambda+2\mu)\int_{\partial\Omega}\sigma^m\div u\,(\dot{u}\cdot n)ds - (\lambda+2\mu)\int\sigma^m\div u\,\div \dot{u}dx  \nonumber\\
%& = (\lambda+2\mu)\int_{\partial\Omega}\sigma^m\div u\,(u\cdot\nabla u\cdot n)ds - \frac{\lambda+2\mu}{2}\left(\int\sigma^{m}(\div u)^{2}dx\right)_{t} \nonumber\\ &\quad - (\lambda+2\mu)\int\sigma^m\div u\,\div(u\cdot\nabla u)dx + \frac{m(\lambda+2\mu)}{2}\sigma^{m-1}\sigma'\int(\div u)^{2}dx \nonumber\\
& \leq (\lambda+2\mu)\int_{\partial\Omega}\sigma^m\div u\,(u\cdot\nabla u\cdot n)ds - \frac{\lambda+2\mu}{2}\left(\int\sigma^{m}(\div u)^{2}dx\right)_{t} \nonumber\\
&\quad C\sigma^{m} \|\nabla u \|_{L^3}^3+ Cm\sigma^{m-1}\sigma'\|\nabla u\|_{L^2}^2
\end{align}
For the first term on the righthand side of \eqref{I20}, by \eqref{flux} and Lemma \ref{lem-f-td}, we obtain
\begin{align}\label{I21}
& (\lambda+2\mu)\int_{\partial\Omega}\sigma^m\div u\,(u\cdot\nabla u\cdot n)ds \nonumber \\
= & -\int_{\partial\Omega}\sigma^m Fu\cdot\nabla n\cdot uds-\int_{\partial\Omega}\sigma^m(P-P_s)u\cdot\nabla n\cdot uds\nonumber \\
\leq & C\sigma^m(\|\nabla F\|_{L^{2}}\|u\|_{L^{4}}^{2}+\|F\|_{L^{6}}\|u\|_{L^{3}}\|\nabla u\|_{L^{2}}+\|F\|_{L^{2}}\| u\|^2_{L^{4}})+C\sigma^m\|\nabla u\|_{L^{2}}^2\nonumber\\
%\leq & C\sigma^m(\|\rho\dot{u}\|_{L^2}+\|\rho\nabla(\Phi-\Phi_s)\|_{L^2}+\|(\rho-\rho_s)\nabla\Phi_s\|_{L^2}+\|\nabla u\|_{L^2}+\|P-P_s\|_{L^2})\|\nabla u\|_{L^{2}}^2+C\sigma^m\|\nabla u\|_{L^{2}}^2\nonumber \\
\leq & \frac{1}{2}\sigma^m\|\rho\dot{u}\|_{L^2}^2+C\sigma^m\|\nabla u\|_{L^{2}}^2(\|\nabla u\|_{L^{2}}^2+1),
\end{align}
Therefore,
\begin{align}\label{I2}
\displaystyle  I_2 \leq &  - \frac{\lambda+2\mu}{2}\left(\int\sigma^{m}(\div u)^{2}dx\right)_t+C\sigma^{m}\|\nabla u\|_{L^{3}}^{3}+\frac{1}{4}\sigma^m\|\rho\dot{u}\|_{L^2}^2\nonumber \\
& +C\sigma^m(\|\nabla u\|_{L^{2}}^2+1)\|\nabla u\|_{L^{2}}^2.
\end{align}
Next, by \eqref{navier-b}, a straightforward computation shows that
\begin{align}\label{I3}
\displaystyle  I_3 %& = -\mu\int\sigma^{m}\nabla\times\omega\cdot\dot{u}dx \nonumber\\
%& =  - \mu\int\sigma^{m}\omega\cdot\curl\dot{u}dx \nonumber\\
& = -\frac{\mu}{2}\left(\int\sigma^{m}|\omega|^{2}dx\right)_t + \frac{\mu m}{2}\sigma^{m-1}\sigma'\int|\omega|^{2}dx- \mu\int\sigma^{m}\omega\cdot\curl(u\cdot\nabla u)dx \nonumber\\
%& = -\frac{\mu}{2}\left(\int\sigma^{m}|\omega|^{2}dx\right)_t + \frac{\mu m}{2}\sigma^{m-1}\sigma'\int|\omega|^{2}dx   \nonumber\\ & \quad - \mu\int\sigma^{m}(\nabla u^{i}\times\nabla_i u)\cdot\omega dx + \frac{\mu}{2}\int\sigma^{m}|\omega|^{2}\,\div udx\nonumber\\
& \leq -\frac{\mu}{2}\left(\int\sigma^{m}|\omega|^{2}dx\right)_t + C\sigma^{m}\|\nabla u\|_{L^{2}}^{2} + C\sigma^{m}\|\nabla u\|_{L^{3}}^{3}.
\end{align}
In view of \eqref{phi-hk} and \eqref{phi-dt2}, it follows that
\begin{align}\label{I4}
I_4%= & \int\sigma^m \rho u_t\cdot\nabla(\Phi-\Phi_s) dx+\int\sigma^m \rho u \cdot \nabla u \cdot \nabla(\Phi-\Phi_s)dx \nonumber \\
  = & \left(\int\sigma^m \rho u\cdot\nabla(\Phi-\Phi_s) dx\right)_t-m\sigma^{m-1}\sigma'\int \rho u\cdot\nabla(\Phi-\Phi_s)dx \nonumber \\
   &-\int\sigma^m \rho u\cdot \nabla (\Phi-\Phi_s)_t dx
     -\int\sigma^m \rho u\cdot \nabla^2 (\Phi-\Phi_s) \cdot u dx\nonumber \\
 \leq & \left(\int\sigma^m \rho u\cdot\nabla(\Phi-\Phi_s) dx\right)_t+Cm\sigma^{m-1}\sigma'\|\nabla (\Phi-\Phi_s)\|_{L^2}^2+C\sigma^m \|\nabla u\|_{L^2}^2 ,
\end{align}
and similarly,
\begin{align}\label{I5}
I_5 %= & \left(\int\sigma^m (\rho-\rho_s) u\cdot\nabla\Phi_s dx\right)_t-m\sigma^{m-1}\sigma'\int (\rho-\rho_s) u\cdot\nabla\Phi_s dx \nonumber \\      &-\int\sigma^m (\rho-\rho_s)_t u\cdot \nabla \Phi_s dx +\int\sigma^m (\rho-\rho_s) u \cdot \nabla u \cdot \nabla\Phi_s dx \nonumber \\  = & \left(\int\sigma^m (\rho-\rho_s) u\cdot\nabla\Phi_s dx\right)_t-m\sigma^{m-1}\sigma'\int (\rho-\rho_s) u\cdot\nabla\Phi_s dx \nonumber \\      &-\int\sigma^m (\rho-\rho_s) u\cdot \nabla^2 \Phi_s \cdot u dx +\int\sigma^m \div (\rho_s u)u \cdot \nabla\Phi_s dx \nonumber \\
 \leq & \left(\int\sigma^m (\rho-\rho_s) u\cdot\nabla\Phi_s dx\right)_t+Cm\sigma^{m-1}\sigma'\|\rho-\rho_s\|_{L^2}^2+C\sigma^m \|\nabla u\|_{L^2}^2,
\end{align}
Making use of the results \eqref{I1}, \eqref{I2},\eqref{I3} and \eqref{I4}, it follows from $\eqref{I0}$ that
\begin{align}\label{I01}
&\left((\lambda+2\mu)\int\sigma^{m}(\div u)^{2}dx+\mu\int\sigma^{m}|\omega|^{2}dx\right)_{t}+\int\sigma^{m}\rho|\dot{u}|^{2}dx \nonumber \\
\leq &\left(\int\sigma^{m}\Big((P-P_s)\,\div u+(\rho-\rho_s) u \cdot\nabla\Phi_s+\rho u \cdot\nabla(\Phi-\Phi_s)\Big)dx\right)_{t}+C\sigma^{m}\|\nabla u\|_{L^{3}}^{3}\nonumber \\
     & +C\sigma^{m}\|\nabla u\|_{L^2}^2(1+\|\nabla u\|_{L^2}^2)+Cm\sigma^{m-1}\sigma'(\|\nabla (\Phi-\Phi_s)\|_{L^2}^{2}+\|\rho-\rho_s\|_{L^2}^2),
\end{align}
integrating over $(0,T]$, choosing $m=1$, by \eqref{tdu1}, \eqref{key1}, Lemma \ref{lem-basic} and Young's inequality, we conclude that %for any $m>0$,
%\begin{align}\label{I02}\displaystyle  &\sigma^{m}\|\nabla u\|_{L^{2}}^{2}+\int_0^T\int\sigma^{m}\rho|\dot{u}|^{2}dxdt \nonumber \\\leq & C\int_0^T m\sigma^{m-1}\sigma'(\|\nabla (\Phi-\Phi_s)\|_{L^2}^{2}+\|\rho-\rho_s\|_{L^2}^2) dt\nonumber \\     &+C\int_0^T\sigma^{m}\|\nabla u\|_{L^{2}}^2(\|\nabla u\|_{L^{2}}^2+1)dt+C\int_0^T\sigma^{m}\|\nabla u\|_{L^{3}}^{3}dt.\end{align}
%Choose $m=1,$ together with \eqref{key1} and \eqref{basic1}, we obtain
\begin{align}\label{a01}
 \displaystyle  \sup_{0\leq t<T}\sigma\|\nabla u\|_{L^{2}}^{2}+\int_0^T \sigma\|\sqrt{\rho}\dot{u}\|_{L^2}^{2}dt\le  C C_0 + C\int_0^{T}\sigma\|\nabla u\|_{L^3}^3 dt.
 \end{align}

Next, operating $ \sigma^{m}\dot{u}^{j}[\pa/\pa t+\div
(u\cdot)] $ to $ \eqref{NSP-2}^j,$ summing with respect to $j$, and integrating over $\Omega,$ together with $ \eqref{NSP}_1 $, we get
\begin{align}\label{J0}
\displaystyle  &\left(\frac{\sigma^{m}}{2}\int\rho|\dot{u}|^{2}dx\right)_t-\frac{m}{2}\sigma^{m-1}\sigma'\int\rho|\dot{u}|^{2}dx \nonumber \\
& = \int\sigma^{m}(\dot{u}\cdot\nabla F_t+\dot{u}^{j}\,\div(u\partial_jF))dx \nonumber \\
&\quad+\mu\int\sigma^{m}(-\dot{u}\cdot\nabla\times\omega_t-\dot{u}^{j}\div((\nabla\times\omega)^j\,u))dx \nonumber \\
&\quad+\int\sigma^{m}(\dot{u}\cdot\nabla\Phi_s(\rho-\rho_s)_t+\dot{u}^{j}\div((\rho-\rho_s)\partial_j \Phi_s\,u))dx \nonumber \\
&\quad+\int\sigma^{m}(\dot{u}\cdot(\rho \nabla(\Phi-\Phi_s))_t+\dot{u}^{j}\div(\rho \partial_j(\Phi-\Phi_s)\,u))dx %\nonumber \\&
\triangleq \sum_{i=1}^4 J_i.
\end{align}
Let us estimate $J_1, J_2, J_3$ and $J_4$.
By \eqref{navier-b} and \eqref{Pu1}, a direct computation yields
\begin{align}\label{J10}
J_1 %& =\int\sigma^{m}\dot{u}\cdot\nabla F_tdx+\int\sigma^{m}\dot{u}^{j}\div(u\partial_jF)dx \nonumber \\
%& = \int_{\partial\Omega}\sigma^{m}F_t\dot{u}\cdot nds-\int\sigma^{m}F_t\,\div\dot{u}dx-\int\sigma^{m}u \cdot \nabla\dot{u}^j\partial_jFdx \nonumber \\
& = \int_{\partial\Omega}\sigma^{m}F_t\dot{u}\cdot nds - (\lambda+2\mu)\int\sigma^{m}(\div\dot{u})^{2}dx + (\lambda+2\mu)\int\sigma^{m}\div\dot{u}\,\nabla u:\nabla udx \nonumber\\
& \quad -\gamma\int\sigma^{m} P\div\dot{u}\,\div udx+\int\sigma^{m}\div \dot{u}\, u \cdot \nabla F dx-\int\sigma^{m}u \cdot\nabla\dot{u} \cdot \nabla F dx \nonumber \\
 &\quad -\int\sigma^{m}\div\dot{u}\,u \cdot \nabla P_s dx \nonumber\\
& \leq \int_{\partial\Omega}\sigma^{m}F_t\dot{u}\cdot nds - (\lambda+2\mu)\int\sigma^{m}(\div\dot{u})^{2}dx + \frac{\delta}{12}\sigma^{m}\|\nabla\dot{u}\|\ltwo+C\sigma^m \|\nabla u\|^4_{L^4}\nonumber\\
& \quad+C\sigma^m \big(\|\nabla u\|_{L^2}^2 \|\nabla F\|_{L^3}^2+ \|\nabla u\|_{L^2}^2\big)
\end{align}
where in the first equality we have used
\begin{align*}
\displaystyle F_t %&=(2\mu+\lm)\div u_t-(P-P_s)_t\\ &=(2\mu+\lm)\div\dot  u-(2\mu+\lm)\div(u\cdot\na u) +u\cdot\na (P-P_s)+\ga P\div u+u \cdot \nabla P_s\\
&=(2\mu+\lm)\div\dot  u-(2\mu+\lm)\na u:\na u  - u\cdot\na F+\ga P\div u+u \cdot \nabla P_s.
\end{align*}
%It is necessary to estimate the three boundary terms in the last equality and we will use \eqref{key1}, Lemma \ref{lem-f-td}, Lemma \ref{lem-ud}, Sobolev trace theorem and Young's, H\"{o}lder's, Poincar\'{e}'s, Sobolev's inequalities.Denote $h\triangleq u\cdot(\nabla n+\nabla n^{*})$ and $u^\bot=-u\times n$.
For the first term on the righthand side of \eqref{J10}, we have
\begin{align}\label{J11}
&\int_{\partial\Omega}\sigma^{m}F_t\dot{u}\cdot nds=-\int_{\partial\Omega}\sigma^{m}F_t\,(u\cdot\nabla n\cdot u)ds \nonumber\\
= & -\left(\int_{\partial\Omega}\sigma^{m}(u\cdot\nabla n\cdot u)Fds\right)_t+m\sigma^{m-1}\sigma'\int_{\partial\Omega}(u\cdot\nabla n\cdot u)Fds \nonumber\\
&\quad +\int_{\partial\Omega}\sigma^{m}\big(F\dot{u}\cdot\nabla n \cdot u+Fu\cdot\nabla n \cdot\dot{u}\big)ds \nonumber \\
&\quad  -\int_{\partial\Omega}\sigma^{m}\big(F(u \cdot \nabla) u\cdot\nabla n \cdot u+Fu\cdot\nabla n \cdot(u \cdot \nabla) u\big)ds\nonumber\\
\leq& -\left(\int_{\partial\Omega}\sigma^{m}(u\cdot\nabla n\cdot u)Fds\right)_t+Cm\sigma^{m-1}\sigma'\|\nabla u\|_{L^2}^{2}\|F\|_{H^1}+\frac{\delta}{12}\sigma^{m}\|\nabla\dot{u}\|_{L^2}^2 \nonumber\\
&+C\sigma^{m}\|\nabla u\|_{L^2}^{4}+\sigma^{m}\|\nabla u\|_{L^2}^{2}\|F\|_{H^1}^{2} +C\sigma^{m}\|\na F\|_{L^6} \|\na u\|^3_{L^2}+ C\sigma^{m}\|\na u\|^2_{L^4}.
\end{align}
Combining these estimates of the boundary terms with \eqref{J10}, %we have
%\begin{align}\label{J14}& J_1 \leq - (\lambda+2\mu)\int\sigma^{m}(\div\dot{u})^{2}dx -\left(\int_{\partial\Omega}\sigma^{m}(u\cdot\nabla n\cdot u)Fds\right)_t\nonumber\\&\quad+Cm\sigma^{m-1}\sigma'\|\nabla u\|_{L^2}^{2}\|F\|_{H^1} +\frac{\delta}{6}\sigma^{m}\|\nabla\dot{u}\|_{L^2}^2+C\sigma^{m}\|\nabla u\|_{L^4}^{4}+C\sigma^{m}\|\nabla u\|_{L^2}^{4}\nonumber\\&\quad +C\sigma^{m}(\|\nabla u\|_{L^2}^{2}\|F\|_{H^1}^{2}+\|\na F\|_{L^6} \|\na u\|^3_{L^2}+\|\nabla u\|_{L^2}^2 \|\nabla F\|_{L^3}^2+\|\nabla u\|_{L^2}^2).\end{align}
%Then, Lemma \ref{lem-f-td} and Lemma \ref{lem-ud} yield that
%\begin{align} \displaystyle \|F\|_{H^1}\leq C(\|\rho \dot{u}\|_{L^2}+\|\nabla(\Phi-\Phi_s)\|_{L^2}+\|\nabla u\|_{L^2}+\|\rho-\rho_s\|_{L^2}),\end{align}
%\begin{align}\displaystyle \|\nabla F\|_{L^6}\leq C(\|\rho\dot{u}\|_{L^6}+\|\nabla(\Phi-\Phi_s)\|_{L^6}+\|\rho-\rho_s\|_{L^6}),\end{align}
%which together with
using \eqref{tdf1},\eqref{udot} and \eqref{tdudot} gives
\begin{align}\label{J1}
& J_1 \leq - (\lambda+2\mu)\int\sigma^{m}(\div\dot{u})^{2}dx -\left(\int_{\partial\Omega}\sigma^{m}(u\cdot\nabla n\cdot u)Fds\right)_t\nonumber\\
& \quad + \frac{\delta}{4}\sigma^{m}\|\nabla\dot{u}\|\ltwo+C\sigma^m \|\nabla u\|^4_{L^4}+C\sigma^{m}\|\sqrt{\rho}\dot{u}\|_{L^2}^{2}(\|\nabla u\|_{L^2}^{2}+\|\nabla u\|_{L^2}^{4})\nonumber\\
&\quad +C\sigma^m\|\nabla u\|_{L^2}^2(1+\|\nabla u\|_{L^2}^2+\|\nabla u\|_{L^2}^4)\nonumber\\
&\quad+Cm\sigma^{m-1}\sigma'(\|\sqrt{\rho}\dot{u}\|_{L^2}^{2}+\|\nabla u\|_{L^2}^{2}+\|\nabla u\|_{L^2}^{4}).
\end{align}
Next, by $ \omega_t=\curl \dot u-u\cdot \na \omega-\na u^i\times \pa_iu$ and a straightforward calculation leads to
\begin{align}\label{J20}
J_2%&=-\mu\int\sigma^{m}\dot{u}\cdot(\nabla\times\omega_t)dx-\mu\int\sigma^{m}\dot{u}\cdot(\nabla\times\omega)\,\div udx\nonumber\\&\quad-\mu\int\sigma^{m} u^{i}\dot{u}\cdot\nabla\times(\nabla_i\omega)dx  \nonumber\\
&=-\mu\int\sigma^{m}|\curl \dot{u}|^{2}dx+\mu\int\sigma^{m}(\omega\times\nabla u^i)\cdot\nabla_i\dot{u}dx \nonumber\\
&\quad-\mu\int\sigma^{m}\curl\dot{u}\cdot(\nabla u^i\times\nabla_i u)dx-\mu\int\sigma^{m}\div u\,\omega\cdot\curl \dot{u}dx\nonumber\\
&\leq -\mu\int\sigma^{m}|\curl \dot{u}|^{2}dx+\frac{\delta}{4}\sigma^{m}\|\nabla\dot{u}\|_{L^2}^2+C\sigma^{m}\|\nabla u\|_{L^4}^4.
\end{align}
In view of \eqref{h3-s} and \eqref{udot}, we obtain that
\begin{align}\label{J30}
\displaystyle J_3 &=\int\sigma^{m}((\rho-\rho_s)u \cdot \nabla^2\Phi_s \cdot\dot{u}dx-\int\sigma^{m}\dot{u}\cdot \nabla\Phi_s\div(\rho_s u)dx \nonumber \\
& \leq C\sigma^{m}(\|\dot{u}\|_{L^6} \|\rho-\rho_s\|_{L^2} \|\nabla u\|_{L^2}+\|\nabla\dot{u}\|_{L^2}\|\nabla u\|_{L^2})\|\nabla\Phi_s\|_{H^2} \nonumber \\
%& \quad + C\sigma^{m}(\|\sqrt{\rho}\dot{u}\|_{L^2} \|\nabla(\Phi-\Phi_s)_t\|_{L^2}+\|\nabla^2(\Phi-\Phi_s)\|_{L^2}\|\sqrt{\rho} u\|_{L^3}\|\dot{u}\|_{L^6}) \nonumber \\
& \leq \frac{\delta}{4}\sigma^{m}\|\nabla\dot{u}\|_{L^2}^2%+C\sigma^m\|\sqrt{\rho}\dot{u}\|_{L^2}^2
+C\sigma^m\|\nabla u\|_{L^2}^2+C\sigma^m\|\nabla u\|_{L^2}^4,
\end{align}
and similarly, by \eqref{udot}, \eqref{phi-hk}, \eqref{phi-dt2} and \eqref{phi-dt}, we have
\begin{align}\label{J40}
\displaystyle J_4 &=\int\sigma^{m}\rho\dot{u}\cdot \nabla(\Phi-\Phi_s)_t dx+   \int\sigma^{m}\rho u\cdot \nabla^2(\Phi-\Phi_s)\cdot\dot{u}dx \nonumber \\
& \leq C\sigma^{m}(\|\sqrt{\rho}\dot{u}\|_{L^2} \|\nabla(\Phi-\Phi_s)_t\|_{L^2}+\|\nabla^2(\Phi-\Phi_s)\|_{L^2}\|\sqrt{\rho} u\|_{L^3}\|\dot{u}\|_{L^6}) \nonumber \\
& \leq \frac{\delta}{4}\sigma^{m}\|\nabla\dot{u}\|_{L^2}^2+C\sigma^m(\|\sqrt{\rho}\dot{u}\|_{L^2}^2 +\|\nabla u\|_{L^2}^2).\end{align}
Combining \eqref{J1}, \eqref{J20}, \eqref{J30} and \eqref{J40} with \eqref{J0},  by \eqref{tdudot} and  choosing $\delta$ small enough, we have
\begin{align}\label{J02}
&\quad \left(\sigma^{m}\|\sqrt{\rho}\dot{u}\|_{L^2}^2\right)_t+\sigma^{m}\|\nabla\dot{u}\|_{L^2}^2\nonumber\\
&\leq -\left(\int_{\partial\Omega}\sigma^{m}(u\cdot\nabla n\cdot u)Fds\right)_t+
C\sigma^{m}\|\nabla u\|_{L^4}^4 +C\sigma^{m}\|\sqrt{\rho}\dot{u}\|_{L^2}^{2}\nonumber\\
& \quad +C\sigma^{m}\|\sqrt{\rho}\dot{u}\|_{L^2}^{2}(\|\nabla u\|_{L^2}^{2}+\|\nabla u\|_{L^2}^{4})+C\sigma^m\|\nabla u\|_{L^2}^2(1+\|\nabla u\|_{L^2}^2+\|\nabla u\|_{L^2}^4)\nonumber\\
&\quad+Cm\sigma^{m-1}\sigma'(\|\sqrt{\rho}\dot{u}\|_{L^2}^{2}+\|\nabla u\|_{L^2}^{2}+\|\nabla u\|_{L^2}^{4}).
\end{align}
For the boundary term in the right-hand side of \eqref{J02}, from Lemma \ref{lem-f-td}, we have
\begin{align}\label{J0b1}
\int_{\partial\Omega}(u\cdot\nabla n\cdot u)Fds\leq C\|\nabla u\|_{L^2}^{2}\|F\|_{H^1}\leq\frac{1}{2}\|\sqrt{\rho}\dot{u}\|_{L^2}^2+C(\|\nabla u\|_{L^2}^2+\|\nabla u\|_{L^2}^4).
\end{align}
Now integrating  \eqref{J02} with  $m=3$ over $(0,T)$, using \eqref{key1}, \eqref{a01} and \eqref{J0b1}, yields
\begin{align}\label{a02}
\sup_{0\leq t\leq T}\sigma^{3}\|\sqrt{\rho}\dot{u}\|_{L^2}^2\!+\!\int_0^{T}\!\!\sigma^{3} \|\nabla \dot{u}\|_{L^2}^2dt
    \le   C C_0\!+\! C\int_0^{T}\!\!\big(\sigma\|\sqrt{\rho}\dot{u}\|_{L^2}^2\!+\!\sigma^3 \|\nabla u\|_{L^4}^4\big)dt.
\end{align}
Combining \eqref{a01} with \eqref{a02}, we immediately obtain \eqref{a00}
and complete the proof of Lemma \ref{lem-a0}.
\end{proof}

\begin{lemma}\label{lem-a1} Assume that $(\rho,u,\Phi)$ is a smooth solution of
 \eqref{NSP}-\eqref{boundary} satisfying \eqref{key1} and the initial data condition \eqref{dt2}, then there exist  positive constants $C$ and $\varepsilon_1$ depending only on $\mu,\,\,\lambda,\, \,\gamma,\,\,a,\,\,\rho_s,\,\,\hat{\rho},\,\,\beta,\,\,\Omega$ and $M$ such that
\begin{align}
& \sup_{0\le t\le \si(T)}t^{1-s}\|\nabla u\|_{L^2}^2+\int_0^{\si(T)}t^{1-s}\int\rho|\dot{u}|^2 dxdt\le C(\hat{\rho},M),\label{uv1}\\
&\sup_{0\le t\le \si(T)}t^{2-s}\int
 \rho|\dot{u}|^2 dx + \int_0^{\si(T)}\int t^{2-s}|\nabla\dot{u}|^2dxdt\leq C(\hat{\rho},M),\label{uv2}
\end{align}
provide that $C_0<\varepsilon_1$.
\end{lemma}
\begin{proof} Suppose $w_1(x,t)$, $w_2(x,t)$ and $w_3(x,t)$ solve problems
\begin{equation}\label{fcz1}
\displaystyle \begin{cases}
  Lw_1=0 \,\, &in \,\,\Omega,\\  w_1(x,0)=w_{10}(x)\,\, &in \,\,\Omega,\\w_1\cdot n=0\,\,\text{and}\,\, \curl w_1\times n=0 \,\,&on \,\,\partial\Omega ,
\end{cases}
\end{equation}
\begin{equation}\label{fcz2}
\displaystyle  \begin{cases}
  Lw_2=-\nabla(P-P_s)\,\, &in \,\,\Omega,\\  w_2(x,0)=0\,\, &in \,\,\Omega,\\w_2\cdot n=0\,\,\text{and}\,\, \curl w_2\times n=0 \,\,&on \,\,\partial\Omega ,
\end{cases}
\end{equation}
and
\begin{equation}\label{fcz3}
\displaystyle  \begin{cases}
  Lw_3=\rho \nabla (\Phi-\Phi_s)+(\rho-\rho_s)\nabla\Phi_s\,\, &in \,\,\Omega,\\  w_3(x,0)=0\,\, &in \,\,\Omega,\\w_3\cdot n=0\,\,\text{and}\,\, \curl w_3\times n=0 \,\,&on \,\,\partial\Omega ,
  \end{cases}
\end{equation}
where $Lf\triangleq\rho\dot{f}-\mu\Delta f-(\lambda+\mu)\nabla\div f $.

A similar way as for the proof of \eqref{basic1} shows that
\begin{align}\label{bw1}
\displaystyle  \sup_{0\le t\le \si(T)}\int\rho|w_1|^2dx+\int_0^{\si(T)}\int|\nabla w_1|^2dxdt\le
C\int|w_{10}|^{2}dx  ,
\end{align}
\begin{align}\label{bw2}
\displaystyle  \sup_{0\le t\le \si(T)}\int\rho|w_2|^2dx+\int_0^{\si(T)}\int|\nabla w_2|^2dxdt%\le\sup_{0\le t\le \si(T)}\|P-P_s\|_{L^2}^2
\le CC_0 ,
\end{align}
and
\begin{align}\label{bw3}
\displaystyle & \sup_{0\le t\le \si(T)}\int\rho|w_3|^2dx+\int_0^{\si(T)}\int|\nabla w_3|^2dxdt %\nonumber\\\le&\sup_{0\le t\le \si(T)}(\|\nabla(\Phi-\Phi_s)\|_{L^2}^2+\|\rho-\rho_s\|_{L^2}^2)
 \le C C_0.
\end{align}

Multiplying \eqref{fcz1} by $w_{1t}$ and integrating over $\Omega,$ by \eqref{key1},  Sobolev's and Young's inequalities, we obtain
\begin{align}\label{bw11}
\displaystyle  &\quad \left(\frac{\lambda+2\mu}{2}\int(\div w_1)^{2}dx + \frac{\mu}{2}\int|\curl w_1|^{2}dx\right)_t + \int\rho|\dot{w}_1|^{2}dx \nonumber \\
 %& =\int\rho\dot{w}_1\cdot(u\cdot\nabla w_1)dx \nonumber \\
 &\le C\|\sqrt{\rho}\dot{w}_1\|_{L^{2}}\|\rho^{\frac{1}{3}}u\|_{L^{3}}\|\nabla w_1\|_{L^{6}} %\nonumber \\ &
 \le C_1C_0^{\frac{\delta_0}{3}}(\|\sqrt{\rho}\dot{w}_1\|_{L^{2}}^{2}+ \|\nabla w_1\|_{L^{2}}^{2}),
\end{align}
since it follows from \eqref{fcz1}, with the similar methods to obtain \eqref{tdf2} as in Lemma \eqref{lem-f-td} that
\begin{equation}\label{xbh1}
\displaystyle  \|\nabla w_1\|_{L^{6}}\leq C\| w_1\|_{W^{2,2}}\leq C(\|\rho\dot{w}_1\|_{L^{2}}+\|\nabla w_1\|_{L^{2}}).
\end{equation}
Together \eqref{bw11} with \eqref{bw1}, and by Gronwall's inequality and Lemma \ref{lem-vn}, it yields
\begin{align}\label{bw12}
\displaystyle  \sup_{0\le t\le \si(T)}\|\nabla w_1\|_{L^{2}}^{2}+\int_0^{\sigma(T)}\int\rho|\dot{w}_1|^{2}dxdt\leq C\|\nabla w_{10}\|_{L^{2}}^{2},
\end{align}
and
\begin{align}\label{bw13}
\displaystyle  \sup_{0\le t\le \si(T)}t\|\nabla w_1\|_{L^{2}}^{2}+\int_0^{\sigma(T)}t\int\rho|\dot{w}_1|^{2}dxdt\leq C\|w_{10}\|_{L^{2}}^{2},
\end{align}
provided $C_0<\hat{\varepsilon}_1\triangleq(2C_1)^{-\frac{3}{\delta_0}}$.
Since the solution operator $w_{10}\mapsto w_1(\cdot,t)$ is linear, by the   standard Stein-Weiss interpolation argument \cite{bl1976}, one can deduce from \eqref{bw12} and \eqref{bw13} that for any $\theta\in [s,1],$
\begin{align}\label{bw14}
\displaystyle  \sup_{0\le t\le\si(T)}t^{1-\theta}\|\nabla
w_1\|_{L^2}^2+\int_0^{\si(T)}t^{1-\theta}\int\n|\dot
{w_1}|^2dxdt\leq C\| w_{10}\|_{H^\theta}^2,
\end{align}
with a uniform constant $C$ independent of $\theta.$

Multiplying \eqref{fcz2} by $w_{2t}$ and integrating over $\Omega$ give that
\begin{align}\label{bw21}
&\quad \left(\frac{\lambda+2\mu}{2}\int(\div w_2)^{2}dx+\frac{\mu}{2}\int|\curl w_2|^{2}dx-\int(P-P_s)\div w_2dx\right)_t+\int\rho|\dot{w}_2|^{2}dx \nonumber\\
%& =\int\rho\dot{w}_2\cdot(u\cdot\nabla w_2)dx-\int P_t\div w_2dx \nonumber\\
& =\int\rho\dot{w}_2\cdot(u\cdot\nabla w_2)dx - \frac{1}{\lambda+2\mu}\int(P-P_s)(F_{w_2}\div u+\nabla F_{w_2}\cdot u ) dx\\
& \quad - \frac{1}{2(\lambda+2\mu)}\int(P-P_s)^{2}\div u dx +\gamma\int P\div u\,\div w_2dx +\int u \cdot \nabla P_s\,\div w_2dx\nonumber\\
& \leq C(\|\sqrt{\rho}\dot{w}_2\|_{L^2}\|\rho^{\frac{1}{3}}u\|_{L^3}\|\nabla w_2\|_{L^6}+\|\nabla u\|_{L^2}\|F_{w_2}\|_{L^2}+\|\nabla F_{w_2}\|_{L^2}\|u\|_{L^2})\nonumber\\
&\quad+C(\|P-P_s\|_{L^2}\|\nabla u\|_{L^2}+\|\nabla u\|_{L^2}\|\nabla w_2\|_{L^2})\nonumber\\
%&\leq CC_0^{\frac{\delta_0}{3}}\|\sqrt{\rho}\dot{w}_2\|_{L^2}(\|\sqrt{\rho}\dot{w}_2\|_{L^2}+\|\nabla w_2\|_{L^2}+\|P-P_s\|_{L^2}+\|P-P_s\|_{L^6})\nonumber\\&   \quad +C\|\nabla u\|_{L^2}(\|\sqrt{\rho}\dot{w}_2\|_{L^2}+\|\nabla w_2\|_{L^2}+\|P-P_s\|_{L^2})\nonumber\\  &\quad+C(\|P-P_s\|_{L^2}\|\nabla u\|_{L^2}+\|\nabla u\|_{L^2}\|\nabla w_2\|_{L^2})\nonumber\\
&\leq C_2C_0^{\frac{\delta_0}{3}}\|\sqrt{\rho}\dot{w}_2\|_{L^2}^{2}+\frac{1}{4}\|\sqrt{\rho}\dot{w}_2\|_{L^2}^{2}
+C(\|\nabla w_2\|_{L^2}^2+\|\nabla u\|_{L^2}^2+\|P-P_s\|_{L^6}^{2}),
\end{align}
where we have utilized \eqref{key1}, \eqref{Pu1} and the following estimates
\begin{equation}\label{xbh2}
\|\nabla F_{w_2}\|_{L^{2}}\leq C\|\rho\dot{w}_2\|_{L^{2}}, \quad \|F_{w_2}\|_{L^{2}}\le C (\|\rho\dot{w}_2\|_{L^{2}}+\|\nabla w_2\|_{L^{2}}),
\end{equation}
\begin{equation}\label{xbh4}
\displaystyle \|\nabla w_2\|_{L^{6}}\leq C(\|\rho\dot{w}_2\|_{L^{2}}+\|\nabla w_2\|_{L^{2}}+\|P-\bar{P}\|_{L^{6}}),
\end{equation}
where $F_{w_2}=(\lambda+2\mu)\div w_2-(P-\bar{P}).$
As a result,
\begin{align}\label{bw22}
\displaystyle  & \quad\left((\lambda+2\mu)\|\div w_2\|_{L^{2}}^{2}+\mu\|\curl w_2\|_{L^{2}}^{2}-2\int(P-P_s)\div w_2dx\right)_t
+\int\rho|\dot{w}_2|^{2}dx \nonumber \\
 & \le C\left(\|\nabla w_2\|_{L^{2}}^{2}+\|\nabla u\|_{L^{2}}^{2}+\|P-P_s\|_{L^{2}}^\frac{2}{3}\right) ,
\end{align}
provide that $C_0<\hat{\varepsilon}_2\triangleq(4C_2)^{-\frac{3}{\delta_0}}$.
Integrating \eqref{bw22} over $(0, \sigma(T))$, with \eqref{bw2} and Lemmas \ref{lem-vn}, \ref{lem-basic}, one has
\begin{align}\label{bw23}
\displaystyle  \sup_{0\le t\le \si(T)}\|\nabla w_2\|_{L^{2}}^{2}+\int_0^{\sigma(T)}\int\rho|\dot{w}_2|^{2}dxdt\leq CC_0^{1/3}.
\end{align}

Similarly, multiplying \eqref{fcz3} by $w_{3t}$ and integrating over $\Omega$ give that
\begin{align}\label{bw31}
&\quad \left(\frac{\lambda+2\mu}{2}\int(\div w_3)^{2}dx+\frac{\mu}{2}\int|\curl w_3|^{2}dx\right)_t+\int\rho|\dot{w}_3|^{2}dx \nonumber\\
& =\left(\int(\rho-\rho_s)w_3\cdot \nabla \Phi_s dx\right)_t+\int \div((\rho-\rho_s)u) w_3\cdot \nabla\Phi_s dx \nonumber\\
 &\quad +\int \div(\rho_su) w_3\cdot \nabla\Phi_s dx
  +\int \rho_s \dot{w_3} \cdot \nabla(\Phi-\Phi_s)  dx   \nonumber \\
 &\quad +\int \rho_s u  \cdot \nabla w_3 \cdot \nabla(\Phi-\Phi_s)  dx
 +\int\rho u\cdot\nabla w_3)\cdot dot{w}_3dx \nonumber \\
 & \triangleq \frac{d}{dt}L_0+\sum_{i=1}^5 L_i.
\end{align}
By \eqref{h3-s} and \eqref{g1}, we have
\begin{align}\label{bw3l0}
L_0\leq C\|\nabla \Phi_s\|_{L^3}\|\rho-\rho_s\|_{L^2}\|w_3\|_{L^6}\leq \frac{\mu}{4}\|\nabla w_3\|_{L^2}^2+C C_0.
\end{align}
Using the similar methods to obtain \eqref{tdf2} as in Lemma \eqref{lem-f-td}, it follows from \eqref{fcz3} that
\begin{align}\label{xbh-1}
\displaystyle \|\nabla w_3\|_{L^{6}}&\leq C(\|\rho\dot{w}_3\|_{L^{2}}+\|(\rho-\bar{\rho})\nabla \Phi_s\|_{L^{2}}+\|\rho\nabla(\Phi-\Phi_s)\|_{L^{2}}+\|\rho w_3\|_{L^{2}})\nonumber\\
&\leq C(\|\rho\dot{w}_3\|_{L^{2}}+\|\rho w_3\|_{L^{2}}+CC_0^{\frac{1}{2}}).
\end{align}
Then, by \eqref{h3-s}, \eqref{basic1} and \eqref{xbh-1}, a directly computation yields
\begin{align}\label{bw3l1}
L_1 &=-\int (\rho-\rho_s)(u \cdot \nabla w_3 \cdot \nabla \Phi_s+u \cdot \nabla^2\Phi_s \cdot w_3)dx \nonumber\\
&\leq C\|\rho-\rho_s\|_{L^2}\|u\|_{L^6}(\|w_3\|_{L^6}+\|\nabla w_3\|_{L^6})\|\nabla \Phi_s\|_{H^2} \nonumber \\
%&\leq C\|\rho-\rho_s\|_{L^2}\|\nabla u\|_{L^2}(\|\nabla w_3\|_{L^2}+\|\nabla^2 w_3\|_{L^2}) \nonumber \\
&\leq \frac{1}{4}\|\sqrt{\rho} \dot{w}_3\|_{L^2}^2+C(\|\nabla u\|_{L^2}^2+\|\nabla w_3\|_{L^2}^2+C_0).
\end{align}
similarly, using \eqref{h3-s}, \eqref{phi-hk}, \eqref{phi-dt}, \eqref{basic1} and \eqref{key1}, we obtain
\begin{align}\label{bw3l2}
\sum_{i=2}^{5}L_i \leq& C(|\nabla u\|_{L^2}\|\nabla\Phi_s\|_{L^3}
+\|\nabla \rho_s\|_{L^2}\|\nabla\Phi_s\|_{L^6}\|u\|_{L^6})\|w_3\|_{L^6}\nonumber \\
&+ C\|\sqrt{\rho}\dot{w_3}\|_{L^2}\|\nabla(\Phi-\Phi_s)\|_{L^2}\nonumber \\
&+ C\|\rho^{1/3}u\|_{L^3}\|\nabla w_3\|_{L^6}(\|\nabla(\Phi-\Phi_s)\|_{L^2}+\|\sqrt{\rho}\dot{w_3}\|_{L^2})\nonumber \\
\leq & C(\|\nabla u\|_{L^2}^{2}+\|\nabla w_3\|_{L^2}^2+C_0)+(\frac{1}{4}+CC^{\delta_0/3})\|\sqrt{\rho}\dot{w_3}\|_{L^2}^2.
\end{align}
Putting \eqref{bw3l0}, \eqref{bw3l1} and \eqref{bw3l2} into \eqref{bw31} yields
\begin{align}\label{bw32}
&\quad \left(\|\nabla w_3\|\ltwo \right)_t+\int\rho|\dot{w}_3|^{2}dx \nonumber\\
&\leq C_3C^{\delta_0/3}\|\sqrt{\rho}\dot{w_3}\|_{L^2}^2+C(\|\nabla w_3\|_{L^2}^2+\|\nabla w_3\|_{L^2}^2+C_0).
\end{align}
Thus, integrating it over $(0,\sigma(T)]$, choosing $C_0\leq \hat{\ve}_3 \triangleq (2C_3)^{-3/\delta_0}$, we obtain
\begin{align}\label{bw34}
\quad \sup_{0\le t\le  \si(T)}\|\nabla w_3\|\ltwo+\int_0^{\sigma(T)}\int\rho|\dot{w}_3|^{2}dxdt \leq CC_0.
\end{align}
%provided $C_0\leq \hat{\ve}_4 \triangleq \{\hat{\ve}_3, (2C_7)^{-2/\delta_0}\}$.
Now let $w_{10}=u_0$, so that $w_1+w_2+w_3=u$,
we derive \eqref{uv1} from \eqref{bw14}, \eqref{bw23} and \eqref{bw34} directly under certain condition $C_0<\varepsilon_1\triangleq\min\{\hat{\varepsilon}_1, \hat{\varepsilon}_2, \hat{\varepsilon}_3\}$.

In order to prove \eqref{uv2},  taking $m=2-s$ in \eqref{J02},  and integrating over $(0,\sigma(T)]$ instead of $(0,T]$, in a similar way as we have gotten \eqref{a02}, we obtain
\begin{align}\label{bu1}
&\sup_{0\le t\le \sigma(T)}\sigma^{2-s}\|\sqrt{\rho}\dot{u}\|_{L^2}^2+\int_0^{\sigma(T)}\sigma^{2-s}\|\nabla\dot{u}\|_{L^2}^2dt \nonumber \\
&\leq C\int_0^{\sigma(T)}\sigma^{2-s} \|\nabla u\|^4_{L^4}dt+C(\hat{\rho}, M).
\end{align}
where we have taken advantage of \eqref{uv1}.
By \eqref{key1} and \eqref{uv1}, we have
\begin{align}\label{bu2}
&\quad\int_0^{\sigma(T)}\sigma^{2-s} \|\nabla u\|^4_{L^4}dt \nonumber \\
&\le C\int_0^{\sigma(T)}t^{2-s}\|\sqrt{\rho}\dot{u}\|_{L^{2}}^{3}\|\nabla u\|_{L^{2}}dt+C\int_0^{\sigma(T)}t^{2-s}\|\nabla u\|_{L^{2}}^4dt+C \nonumber \\
&\le C\int_0^{\sigma(T)}t^{\frac{2s-1}{2}}(t^{1-s}\|\nabla u\|_{L^2}^2)^{\frac{1}{2}}(t^{2-s}\|\sqrt{\rho}\dot{u}\|_{L^{2}}^{2})^{\frac{1}{2}}(t^{1-s}\|\sqrt{\rho}\dot{u}\|_{L^2}^2)dt+C \nonumber \\
&\le C(\hat{\rho}, M)\left(\sup_{0\le t\le  \si(T)}t^{2-s}\|\sqrt{\rho}\dot{u}\|_{L^2}^2\right)^{\frac{1}{2}}+C,
\end{align}
which together with \eqref{bu1} yields \eqref{uv2}.
The proof of Lemma \ref{lem-a1} is completed.
\end{proof}

\begin{lemma}\label{lem-a3} If $(\rho,u,\Phi)$ is a smooth solution of \eqref{NSP}-\eqref{boundary} satisfying \eqref{key1} and the initial data condition \eqref{dt2}, then there exists a positive constant  $\varepsilon_2$   depending only on $\mu ,  \lambda ,  \ga ,  a ,  \on, \hat{\rho}, \beta,  \Omega$, and $M$ such that
\begin{align}\label{ba3}
\displaystyle  A_3(\sigma(T))\leq C_0^{\delta_0},
\end{align}
provided $C_0<\varepsilon_2$.
\end{lemma}

\begin{proof}
Multiplying $\eqref{NSP}_2$ by $3|u|u$, and integrating the resulting equation over $ \O$, lead  to
%\begin{align}\label{ba31} & 3\int|u|\rho \dot{u}\cdot u dx -3(\lambda+2\mu)\int|u|\nabla\div u\cdot u dx- 3\mu\int|u|\nabla\times\omega\cdot udx \nonumber \\&+3\int|u|u\cdot\nabla(P-P_s)dx-3\int(\rho \nabla(\Phi-\Phi_s)+(\rho-\rho_s)\nabla\Phi_s)\cdot |u|udx=0,\end{align} which implies that
\begin{align}\label{ba32}
& \left(\int \n |u|^{3}dx\right)_t+3(\lambda+2\mu)\int\div u\,\div(|u|u)dx+ 3\mu\int\omega\cdot\curl(|u|u)dx\nonumber \\
&- 3\int(P-P_s)\div(|u|u)dx -3\int(\rho \nabla(\Phi-\Phi_s)+(\rho-\rho_s)\nabla\Phi_s)\cdot |u|udx=0.
\end{align}
By \eqref{h3-s}, \eqref{tdu2}, \eqref{phi-hk}, \eqref{key1} and \eqref{basic1}, it follows that
\begin{align}\label{ba33}
&\quad \left(\int\rho|u|^{3}dx\right)_t \nonumber \\
%&\leq C\int|u||\nabla u|^{2}dx+C\int|P-P_s||u||\nabla u|dx\nonumber \\&\quad +C\int\rho|\nabla(\Phi-\Phi_s)||u|^2dx+C\int|\rho-\rho_s||\nabla\Phi_s||u|^2dx \nonumber \\
&\leq C\|u\|_{L^6}\|\nabla u\|_{L^2}\|\nabla u\|_{L^3}+C\|P-P_s\|_{L^3}\|u\|_{L^6}\|\nabla u\|_{L^2}\nonumber \\
&\quad +C\|\rho^{\frac{1}{3}}u\|_{L^3}^2\|\nabla(\Phi-\Phi_s)\|_{L^3}
 +\|\rho-\rho_s\|_{L^2}\|\nabla u\|_{L^2}^2\|\nabla\Phi_s\|_{L^6} \nonumber \\
&\leq C\|\nabla u\|_{L^2}^{\frac{5}{2}}\|\rho\dot{u}\|_{L^2}^{\frac{1}{2}} +C\|\nabla u\|_{L^2}^3+CC_0^{\frac{1}{12}}\|\nabla u\|_{L^2}^{\frac{5}{2}}  \nonumber \\
&\quad+CC_0^{\frac{1}{3}}\|\nabla u\|_{L^2}^2+CC_0^{\frac{2 \delta_0}{3}+\frac{1}{2}}.
\end{align}
Hence, integrating \eqref{ba33} over $(0,\sigma(T))$ and using \eqref{key1}, we get
\begin{align}\label{ba34}
&\quad\sup_{0\le t\le  \si(T) }\int\rho|u|^{3}dx \nonumber\\
&\leq C \int_0^{\sigma(T)}\|\nabla u\|_{L^2}^{\frac{5}{2}}\|\rho\dot{u}\|_{L^2}^{\frac{1}{2}}dt+C\int_0^{\sigma(T)}\|\nabla u\|_{L^2}^{3}dt \nonumber \\
&\quad+CC_0^{\frac{1}{12}}\int_0^{\sigma(T)}\|\nabla u\|_{L^2}^{\frac{5}{2}}dt+CC_0^{2\delta_0}+\int\rho_0|u_0|^3dx \nonumber \\
&\leq C\int_0^{\sigma(T)}(t^{1-\beta}\|\rho\dot{u}\|_{L^2}^2)^{\frac{1}{4}} (t^{1-\beta}\|\nabla u\|_{L^2}^2)^{\frac{5-8 \delta_0}{4}}\|\nabla u\|_{L^2}^{4\delta_0} t^\frac{(\beta-1)(4\delta_0-3)}{2} dt \nonumber \\
&\quad +C\int_0^{\sigma(T)}(t\|\rho\dot{u}\|_{L^2}^2)^{\frac{1}{2}} t^{1-\beta}\|\nabla u\|_{L^2}^2 t^\frac{\beta-2}{2} dt \nonumber \\
&\quad + CC_0^{\frac{1}{12}}\int_0^{\sigma(T)}(t^{1-\beta}\|\nabla u\|_{L^2}^2)^{\frac{5}{4}} t^\frac{5(\beta-1)}{4} dt+CC_0^{2\delta_0}+\int\rho_0|u_0|^3dx \nonumber \\
&\leq CC_0^{2 \delta_0}+\int\rho_0|u_0|^3dx\leq C_4C_0^{2\delta_0},
\end{align}
where we have used the fact $\delta_0=\frac{2\beta-1}{4\beta}\in(0,\frac{1}{4}]$, $\beta\in (1/2,1]$ and
\begin{align}\label{ba35}
 \displaystyle \int\rho_0|u_0|^{3}dx\leq C\|\rho_0^{\frac{1}{2}}u_0\|_{L^{2}}^{{3(2\beta-1)}/{2\beta}}\|u_0\|_{H^\beta}^{{3}/{2\beta}}\leq CC_0^{3 \delta_0}.
 \end{align}
Finally, set $\varepsilon_2\triangleq\min\{\varepsilon_1,(C_4)^{-\frac{1}{\delta_0}}\}$, we get \eqref{ba3} and the proof of Lemma \ref{lem-a3} is completed.
\end{proof}

\begin{lemma}\label{lem-a1a2} Let $(\rho,u,\Phi)$ be a smooth solution of
 \eqref{NSP}-\eqref{boundary} on $\O \times (0,T] $ satisfying \eqref{key1} and the initial data condition $\|u_0\|_{H^\beta}\leq M$ in \eqref{dt2}. Then there exists a positive constant $C$ and $\varepsilon_3$ depending only  on $\mu,$  $\lambda,$   $\gamma,$ $a$, $\beta$, $\on$, $\hat{\rho}$, $M$ and $\Omega$ such that
 \begin{align}
A_1(\sigma(T))+B[0,\sigma(T)]\le & CC_0^{\frac{3}{4}},\label{a1b}\\
A_1(T)\le & C_0^{\frac{1}{2}},\label{a1a2}
 \end{align}
provided $C_0\leq\varepsilon_3$.
Furthermore, if $T>1$, then for any $1\leq t_1< t_2\leq T$,
\begin{align}\label{ajf}
 B[t_1,t_2] \leq CC_0^{3/4}+CC_0(t_2-t_1).
\end{align}
\end{lemma}
\begin{proof} The proof proceeds in two steps. First, for $t\in (0,\sigma(T)), T\geq 1$, by \eqref{tdu2}, \eqref{key1} and Lemma \ref{lem-a1}, one can check that
\begin{align}\label{a4}
&\quad\int_0^{\sigma(T)}\sigma \|\nabla u\|^3_{L^3}dt+\int_0^{\sigma(T)}\sigma^{3} \|\nabla u\|^4_{L^4}dt \nonumber \\
&\le C\int_0^{\sigma(T)}t\left(\|\sqrt{\rho}\dot{u}\|_{L^{2}}^{\frac{3}{2}}\|\nabla u\|_{L^{2}}^{\frac{3}{2}}+\|\nabla u\|_{L^{2}}^3+C_0^{\frac{3}{4}}\|\sqrt{\rho}\dot{u}\|_{L^{2}}^{\frac{3}{2}}+C_0^{\frac{1}{2}}\|\nabla u\|_{L^{2}}^{\frac{3}{2}}+C_0\right)dt \nonumber \\
&\quad + C\int_0^{\sigma(T)}t^{3}\left(\|\sqrt{\rho}\dot{u}\|_{L^{2}}^{3}\|\nabla u\|_{L^{2}}+\|\nabla u\|_{L^{2}}^4+C_0^{\frac{1}{2}}\|\sqrt{\rho}\dot{u}\|_{L^{2}}^{3}+C_0\|\nabla u\|_{L^{2}}+C_0\right)dt \nonumber \\
&\le C\int_0^{\sigma(T)}t^{\frac{3\beta-2}{4}}(\|\nabla u\|_{L^2}^2)^{\frac{3}{4}}(t^{2-\beta}\|\sqrt{\rho}\dot{u}\|_{L^{2}}^{2})^{\frac{3}{4}}%+CC_0 \nonumber \\ & \quad
+CC_0^{\frac{3}{4}}\int_0^{\sigma(T)}t^{\frac{3\beta-2}{4}}(t^{2-\beta}\|\sqrt{\rho}\dot{u}\|_{L^2}^2)^{\frac{3}{4}}dt\nonumber \\
&\quad
+ C\int_0^{\sigma(T)}t^{\frac{2\beta-1}{2}}(\|\nabla u\|_{L^2}^2)^{\frac{1}{2}}(t^{3}\|\sqrt{\rho}\dot{u}\|_{L^{2}}^{2})^{\frac{1}{2}}(t^{2-\beta}\|\sqrt{\rho}\dot{u}\|_{L^2}^2)dt+CC_0 \nonumber \\
& \quad +CC_0^{\frac{1}{2}}\int_0^{\sigma(T)}t^{\frac{2\beta-1}{2}}(t^{3}\|\sqrt{\rho}\dot{u}\|_{L^{2}}^{2})^{\frac{1}{2}}(t^{2-\beta}\|\sqrt{\rho}\dot{u}\|_{L^2}^2)dt+CC_0^{\frac{3}{4}}
 \nonumber \\
&\le CC_0^{\frac{3}{4}},
\end{align}
which, along with \eqref{a00}, gives \eqref{a1b}.

Second, for $t\in (\sigma(T),T), T\geq 1$, we shall show that \eqref{a1a2} holds on each small time-interval. It should be pointed out that \eqref{a1b} implies that on $t=1$,
\begin{equation}
\|\nabla u(1)\|_{L^2}^2+\|\sqrt{\rho}\dot{u}(1)\|_{L^2}^2	\leq CC_0^{3/4}.
\end{equation}
Integrating \eqref{I01} and \eqref{J02} over $[1,3]$, summing them up, from \eqref{key1}, \eqref{tdu2} and \eqref{a1b}, we can obtain that there exists a positive $C$, independent of $T$, such that
\begin{align}\label{a03}
&\sup_{1\le t\le 3}(\|\nabla u\|_{L^2}^2+\|\sqrt{\rho}\dot{u}\|_{L^2}^2) +B[1,3] \nonumber \\
\leq &\|\nabla u(1)\|_{L^2}^2+\|\sqrt{\rho}\dot{u}(1)\|_{L^2}^2 +CC_0+C\int_1^{3} \|\nabla u\|_{L^3}^3 dt+C\int_1^{3} \|\nabla u\|_{L^4}^4 dt \nonumber \\
\leq &CC_0^{3/4}+C\int_1^{3}\|\sqrt{\rho}\dot{u}\|_{L^{2}}^{\frac{3}{2}}\|\nabla u\|_{L^{2}}^{\frac{3}{2}} dt+C\int_1^{3} \|\sqrt{\rho}\dot{u}\|_{L^{2}}^{3}\|\nabla u\|_{L^{2}}dt\nonumber \\
\leq &CC_0^{3/4},
\end{align}
For $T\geq 3$, Let $[T]$ be the largest integer less or equal to $T$. For each integer $k=2,3,\cdots,[T]-1$, we introduce the function $\sigma_k(t) \triangleq \sigma(t+1-k)=\min\{1,t+1-k\}$.  Then, for $t\in [k-1,k+1]$, by replacing $\sigma(t)$ with $\sigma_k(t)$ and repeating the process of Lemma \ref{lem-a0}, we obtain that \eqref{I01} and \eqref{J02} still holds with $\sigma_k(t)$ instead of $\sigma(t)$. Therefore, integrating them over $[k-1,k+1]$ respectively, summing them up, from \eqref{key1}, \eqref{tdu2} and \eqref{a1b}, we can obtain that there exists a positive $C$, independent of $k$ and $T$, such that
\begin{align}\label{a04}
&\sup_{k-1\le t\le k+1 }(\sigma_k(t)\|\nabla u\|_{L^2}^2+\sigma^3_k(t)\|\sqrt{\rho}\dot{u}\|_{L^2}^2) +\int_{k-1}^{k+1}\big(\sigma_k(t)\|\sqrt{\rho}\dot{u}\|_{L^2}^2+\sigma^3_k(t)\|\nabla\dot{u}\|_{L^2}^2\big)dt \nonumber \\
\leq &CC_0^{3/4}+C\int_{k-1}^{k+1} \sigma_k(t)\|\nabla u\|_{L^3}^3 dt+C\int_{k-1}^{k+1} \sigma^3_k(t)\|\nabla u\|_{L^4}^4 dt \nonumber \\
\leq &CC_0^{3/4}+C\int_{k-1}^{k+1}\sigma_k(t)\|\sqrt{\rho}\dot{u}\|_{L^{2}}^{\frac{3}{2}}\|\nabla u\|_{L^{2}}^{\frac{3}{2}} dt+C\int_{k-1}^{k+1}\sigma^3_k(t)\|\sqrt{\rho}\dot{u}\|_{L^{2}}^{3}\|\nabla u\|_{L^{2}}dt\nonumber \\
\leq &CC_0^{3/4}+C\sup_{1\leq t\leq T}(\|\nabla u\|_{L^{2}}^4+\|\sqrt{\rho}\dot{u}\|_{L^2}^4)\int_0^T\|\nabla u\|_{L^{2}}^2dt+\frac{1}{2}\int_{k-1}^{k+1}\sigma_k(t)\|\sqrt{\rho}\dot{u}\|_{L^2}^2dt\nonumber \\
\leq &CC_0^{3/4}+ \frac{1}{2}\int_{k-1}^{k+1}\sigma_k(t)\|\sqrt{\rho}\dot{u}\|_{L^2}^2dt,
\end{align}
which implies that
\begin{align}\label{a05}
&\sup_{k-1\le t\le k+1 }(\sigma_k(t)\|\nabla u\|_{L^2}^2+\sigma^3_k(t)\|\sqrt{\rho}\dot{u}\|_{L^2}^2)\nonumber \\ 
&  +\int_{k-1}^{k+1}\big(\sigma_k(t)\|\sqrt{\rho}\dot{u}\|_{L^2}^2+\sigma^3_k(t)\|\nabla\dot{u}\|_{L^2}^2\big)dt\leq CC_0^{3/4},\quad \text{for}\, k=2,\cdots,[T]-1.
\end{align}
Similarly, choosing $\sigma_{[T]}(t)\triangleq \sigma(t+1-[T])=\min\{1,t+1-[T]\}$, one has 
\begin{align}\label{a05-1}
&\sup_{[T]-1\le t\le T }(\sigma_k(t)\|\nabla u\|_{L^2}^2+\sigma^3_k(t)\|\sqrt{\rho}\dot{u}\|_{L^2}^2)\nonumber \\ 
&  +\int_{[T]-1}^{T}\big(\sigma_k(t)\|\sqrt{\rho}\dot{u}\|_{L^2}^2+\sigma^3_k(t)\|\nabla\dot{u}\|_{L^2}^2\big)dt\leq CC_0^{3/4}.
\end{align}
Since $\sigma_k(t)=1$ for $t\in [k,k+1]$ and $\sigma_{[T]}(t)=1$ for $t\in [[T],T]$, \eqref{a05} and \eqref{a05-1} yields that
\begin{align}\label{a06}
\displaystyle \sup_{1\le t\le T }(\|\nabla u\|_{L^2}^2+\|\sqrt{\rho}\dot{u}\|_{L^2}^2)\leq CC_0^{3/4}.
\end{align}
Combining \eqref{a1b} and \eqref{a06} yields that
\begin{equation}\label{a1a2-2} \sup_{0\le t\le T }(\sigma\|\nabla u\|_{L^2}^2+ \sigma ^3\|\sqrt{\rho} \dot{u}\|_{L^2}^2)\leq C_{5}C_0^{3/4}. \end{equation}
Set $\varepsilon_3\triangleq\min\{\varepsilon_2,(C_{5}^{-4}\}$, \eqref{a1a2} holds when $C_0<\varepsilon_3$. Noting that $[0,\infty]=\cup_{k=1}^{\infty}[k-1,k+1]$ and the estimate \eqref{a05} holds for each $k=2,3,\cdots$, we thus conclude that \eqref{a1a2} is valid in the case when $T=\infty$.

Finally, we proceed to prove \eqref{ajf}. Integrating \eqref{I01} and \eqref{J02} over $(t_1,t_2)$ with $1\leq t_1< t_2\leq T$, summing them up, using \eqref{tdu2}, \eqref{a1a2} and \eqref{a06}, yields
\begin{align}\label{a09}
& \quad B[t_1,t_2] \leq CC_0^{3/4}+C\int_{t_1}^{t_2} (\|\nabla u\|_{L^3}^3+\|\nabla u\|_{L^4}^4) dt\nonumber \\
 &  \leq CC_0^{3/4}+C\int_{t_1}^{t_2}\|\sqrt{\rho}\dot{u}\|_{L^{2}}^{\frac{3}{2}}\|\nabla u\|_{L^{2}}^{\frac{3}{2}} dt+C\int_{t_1}^{t_2} \|\sqrt{\rho}\dot{u}\|_{L^{2}}^{3}\|\nabla u\|_{L^{2}}dt \nonumber \\
 &  \leq CC_0^{3/4}+CC_0(t_2-t_1),
\end{align}
and finishes the proof of Lemma \ref{lem-a1a2}.
\end{proof}

We now proceed to derive a uniform (in time) upper bound for the
density, which turns out to be the key to obtain all the higher
order estimates and thus to extend the classical solution globally.
%We will use an approach motivated by the work of \cite{lx2016}, see also \cite{HLX2012}.
\begin{lemma}\label{lem-brho}
There exists a positive constant  $\ve_4$ depending on  $\mu$,  $\lambda$,   $\ga$, $a$, $\on$, $\hat{\rho}$, $\beta$, $ \Omega,$ and $M$  such that, if  $(\rho,u,\Phi)$ is a smooth solution  of \eqref{NSP}-\eqref{boundary} on $\O \times (0,T] $ satisfying \eqref{key1} and the initial data condition \eqref{dt2}, then
 \begin{align}\label{brho}
 \displaystyle  \sup_{0\le t\le T}\|\n(t)\|_{L^\infty}  \le
\frac{7\hat{\rho} }{4}  ,
 \end{align}
provided $C_0\le \ve_4. $
\end{lemma}

\begin{proof}
First, the equation of  mass conservation $\eqref{NSP}_1$ can be equivalently rewritten in the form
\begin{align}\label{rho1}
\displaystyle  D_t \n=g(\rho)+b'(t),
\end{align}
where
\begin{align}
 \displaystyle D_t\rho\triangleq\rho_t+u \cdot\nabla \rho ,\quad
g(\rho)\triangleq-\frac{\rho(P-P_s)}{2\mu+\lambda}  ,
\quad b(t)\triangleq-\frac{1}{2\mu+\lambda} \int_0^t\rho Fdt.
 \end{align}
Naturally, we shall prove our conclusion by Lemma \ref{lem-z}. It is sufficient to check that the function $b(t)$ must verify \eqref{a100} with some suitable constants $N_0$, $N_1$.

For $t\in[0,\sigma(T)],$ one deduces from \eqref{g1}, \eqref{g2}, \eqref{tdf1}, \eqref{tdxd-u1}, \eqref{udot}, \eqref{key1} and Lemmas \ref{lem-basic}, \ref{lem-a1} that for $\delta_0$ as in Proposition \ref{pr1} and for all $0\leq t_1\leq t_2\leq\sigma(T)$,
\begin{align}\label{bl1}
&\quad |b(t_2)-b(t_1)| =\frac{1}{\lambda+2\mu}\left|\int_{t_1}^{t_2}\rho Fdt\right|\le C\int_0^{\sigma(T)}\|F\|_{L^{\infty}}dt \nonumber\\
& \le C\int_0^{\sigma(T)}\|F\|_{L^{6}}^{\frac{1}{2}}\|\nabla F\|_{L^{6}}^{\frac{1}{2}}dt+C\int_0^{\sigma(T)}\|F\|_{L^{2}}dt\nonumber\\
& \le C\int_0^{\sigma(T)}(\|\sqrt{\rho}\dot{u}\|_{L^2}+C_0^{\frac{1}{6}}+\|\nabla u\|_{L^2})^{\frac{1}{2}}(\|\nabla \dot{u}\|_{L^{2}}+\|\nabla u\|_{L^2}^2+C_0^{\frac{1}{6}})^{\frac{1}{2}}dt+CC_0^{\frac{1}{2}}\nonumber\\
& \leq C\int_0^{\sigma(T)}\|\sqrt{\rho}\dot{u}\|_{L^2}^{\frac{1}{2}}\|\nabla \dot{u}\|_{L^{2}}^{\frac{1}{2}}dt
+ C\int_0^{\sigma(T)}\|\sqrt{\rho}\dot{u}\|_{L^2}^{\frac{1}{2}}\|\nabla u\|_{L^{2}}dt\nonumber\\
& \quad + C\int_0^{\sigma(T)}\|\nabla u\|_{L^2}^{\frac{1}{2}}\|\nabla \dot{u}\|_{L^{2}}^{\frac{1}{2}}dt+ CC_0^{\frac{1}{12}}\nonumber\\
&\leq C\int_0^{\sigma(T)}\big(t^{2-\beta}\|\sqrt{\rho}\dot{u}\|_{L^2}^2\big)^{\frac{1-3\delta_0}{4}}\big(t\|\sqrt{\rho}\dot{u}\|_{L^2}^2\big)^{\frac{3\delta_0}{4}}\big(t^{2-\beta}\|\nabla \dot{u}\|_{L^{2}}^2\big)^{\frac{1}{4}}t^{\frac{(\beta-2)(2-3\delta_0)-3\delta_0}{4}}dt\nonumber\\
& \quad +C\int_0^{\sigma(T)}\big(t^{1-\beta}\|\sqrt{\rho}\dot{u}\|_{L^2}^2\big)^{\frac{1}{4}}\big(t\|\nabla u\|_{L^{2}}^2\big)^{\frac{1}{2}}t^{\frac{2\beta-3}{4}}dt \nonumber \\
& \quad +C\int_0^{\sigma(T)}\big(t^{2-\beta}\|\nabla\dot{u}\|_{L^2}^2\big)^{\frac{1}{4}}\big(t\|\nabla u\|_{L^{2}}^2\big)^{\frac{1}{4}}t^{\frac{\beta-3}{4}}dt+CC_0^{\frac{1}{12}}\nonumber \\
& \leq C_6C_0^{\delta_0/4}.
\end{align}
Combining \eqref{bl1} with \eqref{rho1} and choosing $N_1=0$, $N_0=C_6C_0^{\delta_0/4}$, $\bar{\zeta}=\hat{\rho}$ in Lemma \ref{lem-z} give
\begin{align}\label{rho2}
\displaystyle  \sup_{t\in [0,\si(T)]}\|\rho\|_{L^\infty} \le \hat{\rho}
+C_1C_0^{\delta_0/4} \le\frac{3\hat{\rho}}{2},
\end{align}
provided $C_0\le \hat{\ve}_4\triangleq\min\{\varepsilon_3, \left(\frac{\hat{\rho}}{2C_6}\right)^{\frac{4}{\delta_0}}\}. $

On the other hand, for any $1 \le t_1\le t_2\le T ,$ it follows from \eqref{tdf1}, \eqref{udot}, \eqref{key1} and Lemma \ref{lem-basic} that
\begin{align}\label{br1}
& \quad |b(t_2)-b(t_1)| \le C\int_{t_1}^{t_2}\|F\|_{L^{\infty}}dt \nonumber\\
&\le C\int_{t_1}^{t_2}\|F\|_{L^{6}}^{\frac{1}{2}}\|\nabla F\|_{L^{6}}^{\frac{1}{2}}dt+C\int_{t_1}^{t_2}\|F\|_{L^{2}}dt\nonumber\\
&\le CC_0^{\frac{1}{12}}\int_{t_1}^{t_2}(\|\nabla \dot{u}\|_{L^{2}}^{\frac{1}{2}}+1)dt \nonumber\\
&\le C_7C_0^\frac{3}{4}+C_8C_0^\frac{1}{12}(t_2-t_1).
\end{align}
Now we choose $N_0=C_7C_0^{3/4}$, $N_1=C_8C_0^\frac{1}{12}$ in \eqref{a100} and set $\bar\zeta= \frac{3\hat{\rho}}{2}$ in \eqref{a101}. Since for all $  \zeta \geq\bar{\zeta}=\frac{3\hat{\rho}}{2}>\rho_s+1$,
$$ g(\zeta)%=-\frac{ a\zeta}{2\mu+\lambda}(\zeta^{\gamma}-\rho_s^{\gamma})
\le -\frac{a}{\lambda+2\mu}\leq -C_8C_0^\frac{1}{12} = -N_1. $$
Together with \eqref{rho1} and \eqref{br1}, by Lemma \ref{lem-z}, we have
\begin{align}\label{rho3}
\displaystyle \sup_{t\in
[\si(T),T]}\|\rho\|_{L^\infty}\le \frac{ 3\hat \rho }{2} +C_7C_0^{3/4} \le
\frac{7\hat \rho }{4},
\end{align}
provided $C_0\le \ve_4 \triangleq\min\{\hat{\ve}_4, (\frac{ \hat \n }{4C_7})^{4/3}, (\frac{a}{(2\mu+\lambda)C_8})^{12} \}$.
The combination of \eqref{rho2} with \eqref{rho3} completes the
proof of Lemma \ref{lem-brho}.
\end{proof}

%\section{\label{se4} A priori estimates (II): higher order estimates }
The following Lemmas deal with some necessary higher order estimates, which make sure that one can extend the strong solution globally in time. The proofs are similar to the ones in \cite{Hoff1995,HLX2012,lx2016,cl2019}, and are sketched here for
completeness.
From now on, we always assume that the initial energy $C_0\leq \ve_6$, and the positive constant $C $ may depend on $T$, $\mu$, $\lambda$, $a$, $\ga$, $\on,$ $\hat{\rho},$ $\Omega$, $M$, $\rho_s, \Phi_s$ and $g$, where $g\in L^2(\Omega)$ is given as in \eqref{dt3}.

\begin{lemma}\label{lem-x1}
 There exists a positive constant $C,$ such that
\begin{align}
&\sup_{0\le t\le T}(\|\nabla u\|_{L^2}^2+\|\sqrt{\rho}\dot{u}\|_{L^2}^2)+\int_0^T(\|\sqrt{\rho}\dot{u}\|_{L^2}^2+\|\nabla\dot{u}\|_{L^2}^{2})dt\leq C,\label{x1b1}\\
&\sup_{0\le t\le T}(\|\nabla\rho\|_{L^6}+\|u\|_{H^2})+\int_0^T(\|\nabla u\|_{L^\infty}+\|\nabla^{2} u\|_{L^6}^{2})dt\leq C,\label{x2b1}\\
& \sup_{0\le t\le T}\|\sqrt{\rho}u_t\|_{L^2}^2 + \int_0^T\int|\nabla u_t|^2dxdt\le C. \label{x3b}
\end{align}
\end{lemma}
\begin{proof}
First, taking $s=1$ in \eqref{uv1} along with \eqref{a1a2} and \eqref{ajf} gives
\begin{equation}\label{x1b0}
\displaystyle \sup_{t\in[0,T]}\|\nabla u\|_{L^2}^2 + \int_0^{T}\int\rho|\dot{u}|^2dxdt
  \le C.
\end{equation}
Choosing $m=0$ in \eqref{J02}, integrating it over $(0,T)$, by \eqref{phi-hk}, \eqref{basic1}, \eqref{x1b0} and the compatibility condition \eqref{dt3}, we have
\begin{align*}
\sup_{0\le t\le T}\|\sqrt{\rho}\dot{u}\|_{L^2}^2+\int_0^T\|\nabla\dot{u}\|_{L^2}^{2}dt\leq C+C\int_0^T \|\sqrt{\rho}\dot{u}\|_{L^2}^3 dt\leq C+\frac{1}{2}\sup_{0\le t\le T}\|\sqrt{\rho}\dot{u}\|_{L^2}^2,
\end{align*}
which along with \eqref{x1b0} gives \eqref{x1b1}.
Based on the Beale-Kato-Majda type inequality (see Lemma \ref{lem-bkm}), we can derive \eqref{x2b1}, in arguments similar to \cite{cl2019} and we omit the details.
\eqref{x1b1} and \eqref{x2b1} directly yields \eqref{x3b}.
This finishes the proof.
\end{proof}

\begin{lemma}\label{lem-x3}
There exists a positive constant $C$ such that
\begin{align}
& \sup\limits_{0\le t\le T}\left(\|\rho\|_{H^2}\! +\!
 \|P\|_{H^2}\!+\!
   \|\rho_t\|_{H^1}\!+\!\|P_t\|_{H^1}\right)
    \!+\!\! \int_0^T\!\!\left(\|\n_{tt}\|_{L^2}^2\!+\!\|P_{tt}\|_{L^2}^2\right)dt
\le C, \label{x4b} \\
& \sup\limits_{0\le t\le T}\sigma \|\nabla u_t\|_{L^2}^2
    + \int_0^T\sigma\|\sqrt{\rho}u_{tt}\|_{L^2}^2dt
\le C,\label{x4bb}\\
& \sup\limits_{0\le t\le T}\left(\|\nabla(\Phi-\! \Phi_s)\|_{H^3} +
   \|\nabla\Phi_t\|_{H^2}+\|\nabla\Phi_{tt}\|_{L^2}\right)
\le C.\label{x3-phi}
\end{align}
\end{lemma}
\begin{proof} From \eqref{2tdu}, \eqref{3tdu}, \eqref{phi-hk}, \eqref{udot} and Lemma \ref{lem-x1}, we have
\begin{equation}
 \|\nabla^3 u\|_{L^2}\leq C \|\nabla\dot{u}\|_{L^2}+C \|\nabla^2 P\|_{L^2}+C,
\end{equation}
which can help us to get \eqref{x4b} by the same method as that in \cite{cl2019}. So is \eqref{x4bb} and we omit the details.
From \eqref{NSP}, we have
\begin{equation}\label{x-phi}
\Delta \Phi_t=\rho_t,\qquad \Delta\Phi_{tt}=-\div (\rho_t u+\rho u_t).
\end{equation}
With the similar arguments used in Lemma \ref{lem-phi}, combining \eqref{x4b} with \eqref{phi-hk} and \eqref{x-phi} leads to \eqref{x3-phi}.
The proof is completed.
\end{proof}

\begin{lemma}\label{lem-x5}
There exists a positive constant $C$ so that for any $q\in(3,6),$
\begin{align}
& \sup_{t\in[0,T]} \si \|\nabla u\|_{H^2}^2+\int_0^T \left(\|\nabla u\|_{H^2}^2+\|\na^2 u\|^{p_0}_{W^{1,q}}+\si\|\na u_t\|_{H^1}^2\right)dt\le C,\label{x5bb}\\
& \sup_{t\in[0,T]}\left(\|\rho- \rho_s\|_{W^{2,q}} +\|P-P_s\|_{W^{2,q}}+\|\nabla(\Phi-\Phi_s)\|_{W^{3,q}}\right)\le C,\label{x5b}
\end{align}
where $p_0=\frac{9q-6}{10q-12}\in(1,\frac{7}{6}).$
\end{lemma}
\begin{proof}
Let's start with \eqref{x5bb}.  By Lemma \ref{lem-x1}, \eqref{x4bb}, \eqref{x3-phi} and Poincar\'{e}'s, Sobolev's inequalities, one can check that
% \begin{align}\label{x5b1}
%  \|\nabla (\n \dot u) \|_{L^2}&\le
%  \||\nabla \n ||  u_t|  \|_{L^2}\!+ \!\|\n \nabla   u_t  \|_{L^2}\!
%  +\! \||\nabla \n|| u||\nabla u| \|_{L^2}\!+\! \|\n|\nabla  u|^2\|_{L^2}\!+\! \|  \n |u || \nabla^2 u| \|_{L^2}\nonumber \\
% %&\le \|\nabla \n \|_{L^3} \|  u_t  \|_{L^6}+ C\| \nabla   u_t  \|_{L^2} + C\| \nabla \n\|_{L^3}\| u\|_{L^\infty}\|\nabla u \|_{L^6}\nonumber \\&\quad + C\| \nabla  u\|_{L^3}\| \nabla  u\|_{L^6} + C\|     u \|_{L^\infty}\| \nabla^2 u  \|_{L^2}\\
% &\le C+C\| \nabla   u_t  \|_{L^2}.
% \end{align}
% Consequently, together with \eqref{x4bb} and Lemma \ref{lem-x1}, it yields
\begin{align}\label{x5b2}
\|\nabla^2 u\|_{H^1} &\le C (\|\rho \dot u\|_{H^1}\!+\!\|\rho\nabla(\Phi\!-\!\Phi_s)\|_{H^1}\!+\!\|(\rho\!-\!\rho_s)\nabla\Phi_s\|_{H^1}\!+\!\| P\!-\!P_s\|_{H^2}\!+\!\|u\|_{L^2})\nonumber \\
 &\le C+C \|\na  u_t\|_{L^2}.
\end{align}
It then follows from \eqref{x5b2}, \eqref{x3b} and \eqref{x4bb} that
\begin{align}\label{x5b3}
\displaystyle \sup\limits_{0\le
t\le T}\si\|\nabla  u\|_{H^2}^2+\ia \|\nabla  u\|_{H^2}^2dt \le
 C.
\end{align}
Next, we deduce from Lemmas \ref{lem-x1}-\ref{lem-x3} that
\begin{align}\label{x5b4}
\displaystyle  \|\na^2u_t\|_{L^2}
&\le C(\|(\rho\dot{u})_t\|_{L^2}+\|\nabla P_t\|_{L^2}+\|(\rho\nabla\Phi)_t\|_{L^2}+\|u_t\|_{L^2}) \nonumber \\
%&\le C(\|\n  u_{tt}+\n_t u_t+\n_t u\cdot\nabla u + \n u_t\cdot\nabla u+\n u\cdot\nabla u_t\|_{L^2})+C\|\nabla P_t\|_{L^2}\nonumber \\&\quad +C\|u_t\|_{L^2}+C(\|H\|_{L^\infty}\|\nabla H_t\|_{L^2}+\|\nabla H\|_{L^3}\|H_t\|_{L^6})\nonumber\\&\le C\left(\|\n  u_{tt}\|_{L^2}+ \|\n_t\|_{L^3}\|u_t\|_{L^6}+\|\n_t\|_{L^3}\| u\|_{L^\infty}\|\nabla u\|_{L^6}+ \| u_t\|_{L^6}\|\nabla u\|_{L^3}\right)+C\nonumber\\&\quad+C\left(\| u\|_{L^\infty}\|\nabla u_t\|_{L^2}+\|u_t\|_{L^2}+\|H\|_{L^\infty}\|\nabla H_t\|_{L^2}+\|\nabla H\|_{L^3}\|H_t\|_{L^6}\right)\nonumber \\
&\le C\|\sqrt{\rho}  u_{tt}\|_{L^2}+C\|\nabla  u_t\|_{L^2} +C,
\end{align}
where in the first inequality, we have utilized the $L^p$-estimate for the following elliptic system
\begin{equation}\label{x5b5}
\begin{cases}
  \mu\Delta u_t+(\lambda+\mu)\nabla\div u_t=(\rho\dot{u})_t+\nabla P_t+(\rho\nabla\Phi)_t \,\,\, &\text{in} \,\,\Omega,\\
  u_t\cdot n=0\,\,\,\text{and} \,\,\,\omega_t\times n=0\,\,&\text{on} \,\,\partial\Omega.
\end{cases}	
\end{equation}
Together with \eqref{x5b4} and \eqref{x4bb} yields
\begin{align}\label{x5b6}
\displaystyle  \int_0^T\sigma\|\nabla u_t\|_{H^1}^2dt\leq C.
\end{align}
By Sobolev's inequality, \eqref{udot}, \eqref{x2b1} and \eqref{x4bb}, we get for any $q\in (3,6)$,
\begin{align}\label{x5b7}
\displaystyle \|\na(\n\dot u)\|_{L^q}
&\le C \|\na \n\|_{L^q}(\|\nabla\dot{u}\|_{L^q}+\|\nabla\dot{u}\|_{L^2}+\|\nabla u\|_{L^2}^2)+C\|\na\dot u \|_{L^q}\nonumber\\
%&\le C (\|\nabla\dot{u}\|_{L^2}+\|\nabla u\|_{L^2}^2)+C(\|\na u_t \|_{L^q}+\|\na(u\cdot \na u ) \|_{L^q})\nonumber\\
&\le C (\|\nabla u_t\|_{L^2}+1)+C\|\na u_t \|_{L^2}^{\frac{6-q}{2q}}\|\nabla u_t\|_{L^6}^{\frac{3(q-2)}{2q}}\nonumber\\
&\quad+C(\|u \|_{L^\infty}\|\nabla^{2}u\|_{L^q}+\|\nabla u\|_{L^{\infty}}\|\nabla u\|_{L^q})\nonumber\\
&\le C\sigma^{-\frac{1}{2}}+C\|\nabla u\|_{H^2}+C\sigma^{-\frac{1}{2}}(\sigma\|\nabla u_t\|_{H^1}^2)^{\frac{3(q-2)}{4q}}+C.
\end{align}
Integrating this inequality over $[0,T],$ by \eqref{x1b1} and \eqref{x5b6}, we have
\begin{align}\label{x5b8}
\displaystyle  \int_0^T\|\nabla(\rho\dot{u})\|_{L^q}^{p_0}dt\leq C.
\end{align}
%which together with \eqref{x5b3} and \eqref{x5b6} gives \eqref{x5bb}.
On the other hand, the combination of  \eqref{Pu1} with \eqref{2tdu}, \eqref{3tdu}, \eqref{x1b1} and \eqref{x4bb} gives
\begin{align}\label{x5b9}
\displaystyle (\|\na^2 P\|_{L^q})_t & \le C \|\na u\|_{L^\infty} \|\na^2 P\|_{L^q}   +C  \|\na^2 u\|_{W^{1,q}}   \nonumber \\
& \le C (1+\|\na u\|_{L^\infty} )\|\na^2 P\|_{L^q}+C(1+ \|\na  u_t\|_{L^2})+ C\| \na(\n \dot u )\|_{L^{q}},
\end{align}
where in the last inequality we have used the  following simple fact that
\begin{align}\label{x5b10}\displaystyle \|\na^2 u\|_{W^{1,q}}
%&\le C(\|\rho\dot{u}\|_{L^q}+\|\nabla(\rho\dot{u})\|_{L^q}+\|\nabla^{2} P\|_{L^q}+\|\nabla P\|_{L^q}\nonumber \\&\quad+\|\nabla u\|_{L^2}+\|P-P_s\|_{L^2}+\|P-P_s\|_{L^q})\nonumber \\
&\le C(1 + \|\na  u_t\|_{L^2}+ \| \na(\n\dot u )\|_{L^{q}}+\|\na^2  P\|_{L^{q}}),\end{align}
due to \eqref{2tdu}, \eqref{3tdu}, \eqref{x1b1} and \eqref{x4bb}.
Hence, applying Gronwall's inequality in \eqref{x5b9}, we deduce from \eqref{x2b1}, \eqref{x3b}  and \eqref{x5b8} that
\begin{align}\label{x5b11}
\displaystyle  \sup_{t\in[0,T]}\|\nabla^{2}P\|_{L^q}\leq C ,
\end{align}
which along with \eqref{x3b}, \eqref{x4bb} and \eqref{x5b10} also gives
\begin{align}\label{x5b12}
\displaystyle  \sup_{t\in[0,T]}\|P-P_s\|_{W^{2,q}}+\int_0^T\|\nabla^{2}u\|_{W^{1,q}}^{p_0}dt\leq C .
\end{align}
Similarly, one has
\begin{align}
\displaystyle \sup\limits_{0\le t\le T}\|
\n-\rho_s\|_{W^{2,q}} \le
 C,
\end{align}
which together with \eqref{x5b12} and \eqref{phi-hk} gives \eqref{x5b}. The proof of Lemma \ref{lem-x5} is finished.
\end{proof}

\begin{lemma}\label{lem-x6}
There exists a positive constant $C$ such that
\begin{align}\label{x6b}
\displaystyle \sup_{0\le t\le T}\sigma\left(\|\sqrt{\rho} u_{tt}\|_{L^2}+\|\nabla u_t\|_{H^1}+\|\na u\|_{W^{2,q}}\right)+\int_{0}^T\sigma^2\|\nabla u_{tt}\|_{2}^2dt\le C ,
\end{align}
for any $q\in (3,6)$.
\end{lemma}

\begin{proof} Differentiating $\eqref{NSP-2}$ with respect to $t$ twice, multiplying it by $u_{tt}$, and integrating over $\Omega$ lead to
\begin{align}\label{x6b2}
&\quad \frac{d}{dt}\int \frac{\rho}{2}|u_{tt}|^2dx+(\lambda+2\mu)\int(\div u_{tt})^2dx+\mu\int|\omega_{tt}|^2dx \nonumber \\
&=-4\!\int\n u^i_{tt} u\!\cdot\!\na\!  u^i_{tt} dx\!-\!\int_{ }(\n u)_t\!\cdot \left[\na (u_t\!\cdot\! u_{tt})\!+\!\na u_t\!\cdot\! u_{tt}\right]dx\!-\!\int_{}(\n_{tt}u\!+\!2\n_tu_t)\!\cdot\!\na u\!\cdot\! u_{tt}dx \nonumber \\
&\quad -\!\int (\n u_{tt}\cdot\na u\cdot  u_{tt}\!-\!P_{tt}{\rm div}u_{tt})dx\!-\!\int (\rho_{tt}\nabla\Phi\!+\!2\rho_t\nabla\Phi_t\!+\!\rho\nabla\Phi_{tt})u_{tt}dx \nonumber \\
 & \triangleq\sum_{i=1}^5 R_i.
\end{align}
Let us estimate $R_i$ for $i=1,\cdots,5$. H\"{o}lder's inequality, \eqref{x1b1}, \eqref{x2b1}, \eqref{x3b}, \eqref{x4b} and \eqref{x4bb}, give
\begin{align}\label{x6r1}
\displaystyle  \sum_{i=1}^4 R_i \le &
 \de \|\na u_{tt}\|_{L^2}^2+C(\de)\|\sqrt{\rho}u_{tt}\|^2_{L^2}+C(\de)\|\nabla u_t\|_{L^2}^3+C(\de)\|\nabla u_t\|_{L^2}^2\nonumber \\
&+C(\de)\|\n_{tt}\|_{L^2}^2+C(\de)\|P_{tt}\|^2_{L^2},
\end{align}
By \eqref{x1b1}, \eqref{x3b}, \eqref{x4b} and \eqref{x4bb}, we conclude that
\begin{align}
R_5 &\le C\left(\|\n_{tt}\|_{L^2}\|\na \Phi\|_{L^3}+\|\n_{t}\|_{L^3}\|\na\Phi_{t}\|_{L^6}\right)\|u_{tt}\|_{L^6}+\|\nabla\Phi_{tt}\|_{L^2}\|u_{tt}\|_{L^2} \nonumber \\
&\le \de \|\na u_{tt}\|_{L^2}^2+C(\de)(\|\rho_{tt}\|^2_{L^2}+\|\nabla^2\Phi_{t}\|_{L^2}^2+\|\nabla \Phi_{tt}\|^2_{L^2}),\label{x6r5}
\end{align}
Substituting these estimates into \eqref{x6b2}, utilizing the fact that
\begin{align}\label{x6b3}
\displaystyle  \|\nabla u_{tt}\|_{L^2}\leq C(\|\div u_{tt}\|_{L^2}+\|\omega_{tt}\|_{L^2}),
\end{align}
due to Lemma \ref{lem-vn} since $u_{tt}\cdot n=0,$ on $\partial\Omega,$ and then choosing $\de$ small enough, we can get
\begin{align}\label{x6b4}
&\frac{d}{dt}\|\sqrt{\rho}u_{tt}\|^2_{L^2}+\|\na u_{tt}\|_{L^2}^2\nonumber \\
&\le C (\|\sqrt{\rho}u_{tt}\|^2_{L^2}+\|\rho_{tt}\|^2_{L^2}+\|P_{tt}\|^2_{L^2}+\|\nabla u_{t}\|^3_{L^2}+\|\nabla u_{t}\|^2_{L^2}+C),
\end{align}
which together with \eqref{x4b}, \eqref{x4bb}, and by Gronwall's inequality yields that
\begin{align}\label{x6b5}
\sup_{0\leq t\leq T}\sigma^2\|\sqrt{\rho}u_{tt}\|^2_{L^2}+\int_0^T\sigma^2\|\na u_{tt}\|_{L^2}^2dt \leq C.
\end{align}
Furthermore, it follows from \eqref{3tdu}, \eqref{x5b4} and \eqref{x4bb} that
\begin{align}\label{x6b6}
\displaystyle &\quad \sup_{0\le t\le T}\sigma\|\nabla^2 u_t\|_{L^2}\leq C \sigma(1+\|\sqrt{\rho}u_{tt}\|_{L^2}+\|\nabla u_t\|_{L^2}) \leq C.
\end{align}
Finally, we deduce from \eqref{x4bb}, \eqref{x5b}, \eqref{x5bb}, \eqref{x5b7}, \eqref{x5b10}, \eqref{x6b5} and \eqref{x6b6} that
\begin{align}\label{x6b7}
\displaystyle \sigma\|\na^2 u\|_{W^{1,q}}
& \le C\sigma (1+\|\na  u_t\|_{L^2}+\| \na(\n
\dot u )\|_{L^{q}}+\|\na^2  P\|_{L^{q}})\nonumber \\
& \le C(1+ \sigma\|\na u\|_{H^2}+\sigma^{\frac{1}{2}}(\sigma\|\na u_t\|_{H^1}^2)^{\frac{3(q-2)}{4q}})\nonumber \\
&\le C+C\sigma^{\frac{1}{2}}(\sigma^{-1})^{\frac{3(q-2)}{4q}} \le C ,
\end{align}
together with \eqref{x6b5} and \eqref{x6b6} yields \eqref{x6b} and this completes the proof of Lemma \ref{lem-x6}.
\end{proof}

\section{Proof of  Theorem  \ref{th1}}\label{se5}

With all the a priori estimates in Section \ref{se3} at hand, we are going to  prove the main result of the paper in this section.

{\it Proof of Theorem \ref{th1}.}
By Lemma \ref{lem-local}, there exists a
$T_*>0$ such that the  system \eqref{NSP}-\eqref{boundary} has a unique classical solution $(\rho,u,\Phi)$ on $\Omega\times
(0,T_*]$. One may use the a priori estimates, Proposition \ref{pr1} and Lemmas \ref{lem-x3}-\ref{lem-x6} to extend the classical
solution $(\rho,u,\Phi)$ globally in time.

First, by the definition of \eqref{As1}, the assumption of the initial data \eqref{dt2} and \eqref{ba35}, one immediately checks that
\begin{align}\label{pf1}
\displaystyle  0\leq\rho_0\leq \hat{\rho},\,\, A_1(0)=0, \,\,  A_2(0)\leq C_0^{\delta_0}.
\end{align}
Therefore, there exists a $T_1\in(0,T_*]$ such that
\begin{equation}\label{pf2}
0\leq\rho_0\leq2\hat{\rho}, \,\, A_1(T)\leq 2C_0^{\frac{1}{2}}, \,\, A_2(\sigma(T))\leq 2C_0^{\delta_0},
\end{equation}
hold for $T=T_1.$
Next, we set
\begin{align}\label{pf3}
\displaystyle  T^*=\sup\{T\,|\,{\rm \eqref{pf2} \ holds}\}.
\end{align}
Then $T^*\geq T_1>0$. Hence, for any $0<\tau<T\leq T^*$
with $T$ finite, it follows from Lemmas \ref{lem-x1}-\ref{lem-x6}
that
\begin{equation}\label{pf4}
\begin{cases}
(\rho-\rho_s, \nabla(\Phi-\Phi_s) \in C([0,T]; H^2 \cap W^{2,q}), %\quad \nabla(\Phi-\Phi_s) \in C([0,T]; H^2 \cap W^{2,q}),
\\
\nabla u \in C([\tau ,T]; H^1), \quad \nabla u_t\in C([\tau ,T]; L^q);
\end{cases}	
\end{equation}
where one has taken advantage of the standard embedding
$$L^\infty(\tau ,T;H^1)\cap H^1(\tau ,T;H^{-1})\hookrightarrow
C\left([\tau ,T];L^q\right),\quad\mbox{ for any } q\in [2,6).  $$
Due to \eqref{x3b}, \eqref{x4bb}, \eqref{x6b} and $\eqref{NSP}_1$,
we obtain
\begin{align*}
&\quad\int_{\tau}^T \left|\left(\int\n|u_t|^2dx\right)_t\right|dt %\\&
\le\int_{\tau}^T\left(\|  \n_t  |u_t|^2 \|_{L^1}+2\|  \n  u_t\cdot u_{tt} \|_{L^1}\right)dt\\
%&\le C\int_{\tau}^T \left( \| \n|\div u||u_t|^2 \|_{L^2}+\|  |u||\na \n| |u_t|^2 \|_{L^1}+ \|\sqrt{\rho}  u_t \|_{L^2}\|\sqrt{\rho}u_{tt} \|_{L^2}\right)dt\\
&\le C\int_{\tau}^T\left( \| \sqrt{\rho} |u_t|^2 \|_{L^2}\|\na u\|_{L^\infty}+\|  u\|_{L^6}\|\na\n\|_{L^2} \|u_t  \|^2_{L^6}+  \|\sqrt{\rho}u_{tt} \|_{L^2}\right)dt%\\&
\le C,
\end{align*}
which together with \eqref{pf4} yields
\begin{align}\label{pf5}
\displaystyle  \sqrt{\rho}u_t, \quad\sqrt{\rho}\dot u \in C([\tau,T];L^2).
\end{align}
Finally, we claim that
\begin{align}\label{pf6}
 \displaystyle  T^*=\infty.
 \end{align}
Otherwise, $T^*<\infty$. Then by Proposition \ref{pr1}, it holds that
\begin{equation}\label{pf7}
0\leq\rho\leq\frac{7}{4}\hat{\rho},\,\,A_1(T^*)\leq C_0^{\frac{1}{2}},\,\, A_2(\sigma(T^*))\leq C_0^{\delta_0},	
\end{equation}
It follows from Lemmas \ref{lem-x5}, \ref{lem-x6} and \eqref{pf5} that $(\rho(x,T^*),u(x,T^*), H(x,T^*))$ satisfies the initial data condition \eqref{dt1}  and \eqref{dt3}, where  $g(x)\triangleq\sqrt{\rho}\dot u(x, T^*),\,\,x\in \Omega.$
Thus, Lemma \ref{lem-local} implies that there exists some $T^{**}>T^*$ such that \eqref{pf2} holds for $T=T^{**}$, which contradicts the definition of $ T^*.$ As a result, \eqref{pf6} holds.
By Lemmas \ref{lem-local} and \ref{lem-x1}-\ref{lem-x6}, it indicates that $(\rho,u,\Phi)$ is in fact the unique classical solution defined on $\Omega\times(0,T]$ for any  $0<T<T^*=\infty.$ The proof of Theorem \ref{th1} is finished.
 \endproof

%{\bf Conflict of Interest:} The authors declare that they have no conflict of interest.

%%      ---------------------------------------------------------------------
%%      ------------------------- APPENDIX (OPTIONAL) -----------------------
%%      ---------------------------------------------------------------------

%%      If you have one appendix, uncomment the line \appendix and add
%%      a \section{ *** APPENDIX TITLE ***}. If you have more than
%%      one, uncomment the line \appendices and add a \section{ ***
%%      APPENDIX TITLE ***} command for each appendix title.

%%      Type body of appendix/-ices here.
\appendix
\section{Some basic theories and lemmas}\label{appendix-a}
In this appendix, we review some elementary inequalities and important lemmas that are used extensively in this paper.

First, we recall the well-known Gagliardo-Nirenberg inequality (see \cite{Nir1959}).
\begin{lemma}[Gagliardo-Nirenberg]\label{lem-gn}
Assume that $\Omega$ is a bounded Lipschitz domain in $\r^3$. For  $p\in [2,6],\,q\in(1,\infty), $ and
$ r\in  (3,\infty),$ there exist two generic
 constants
$C_1,\,\,C_2>0$ which may depend  on $p$, $q$ and $r$ such that for any  $f\in H^1({\O }) $
and $g\in  L^q(\O )\cap D^{1,r}(\O), $
\be\label{g1}\|f\|_{L^p(\O)}\le C_1 \|f\|_{L^2}^{\frac{6-p}{2p}}\|\na
f\|_{L^2}^{\frac{3p-6}{2p}}+C_2\|f\|_{L^2} ,\ee
\be\label{g2}\|g\|_{C\left(\ol{\O }\right)} \le C_1
\|g\|_{L^q}^{q(r-3)/(3r+q(r-3))}\|\na g\|_{L^r}^{3r/(3r+q(r-3))} + C_2\|g\|_{L^2}.
\ee
Moreover, if $f\cdot n|_{\partial\Omega}=0$, $g\cdot n|_{\partial\Omega}=0$, then the constant $C_2=0.$
%Moreover, if either $f\cdot n|_{\partial\Omega}=0$ or $\int_\Omega fdx=0$, either $g\cdot n|_{\partial\Omega}=0$ or $\int_\Omega gdx=0$, then the constant $C_2=0.$
\end{lemma}

% The following trace theorem can be found in \cite{GPG}.
% \begin{lemma}[Trace Theorem]\label{lem-trace}
% Let $\Omega$ be a bounded domain in $\r^n$ with  Lipschitz boundary.
%  For $u\in W^{k,p}(\Omega)$, $ 1\leq p <\infty$. Assume that
% \bnn\begin{cases}
% r\in [p,\frac{(n-1)p}{n-kp}]~~\,\, &if ~kp<n;\\ r\in [p,+\infty]~~\,\,&if~kp\geq n,
% \end{cases}\enn
% then $u\in L^r(\partial\Omega)$ and for any $\delta>0$, there exists a positive constant $C=C(n,r,p,\Omega,\delta)$ such that
% $$\|u\|_{L^r(\partial\Omega)}\leq \delta\|\nabla u\|_{L^p(\Omega)}+C(\delta)\|u\|_{L^p(\Omega)}.$$
% \end{lemma}

In order to get the uniform (in time) upper bound of the density $\n,$ we need the following Zlotnik  inequality in \cite{zlo2000}.
\begin{lemma}\label{lem-z}
Suppose the function $y$ satisfy
\bnn y'(t)= g(y)+b'(t) \mbox{  on  } [0,T] ,\quad y(0)=y^0, \enn
with $ g\in C(R)$ and $y, b\in W^{1,1}(0,T).$ If $g(\infty)=-\infty$
and \be\label{a100} b(t_2) -b(t_1) \le N_0 +N_1(t_2-t_1)\ee for all
$0\le t_1<t_2\le T$
  with some $N_0\ge 0$ and $N_1\ge 0,$ then
\bnn y(t)\le \max\left\{y^0,\overline{\zeta} \right\}+N_0<\infty
\mbox{ on
 } [0,T],
\enn
where $\overline{\zeta} $ is a constant such
that \be\label{a101} g(\zeta)\le -N_1 \quad\mbox{ for }\quad \zeta\ge \overline{\zeta}.\ee
\end{lemma}

% Consider the Lam\'{e}'s system
% \be\label{lame1}\begin{cases}
% -\mu\Delta u-(\lambda+\mu)\nabla\div u=f \,\, &in~ \Omega, \\
% u\cdot n=0\,\,\text{and}\,\,\curl u\times n=0\,\,&on\,\,\partial\Omega,
% \end{cases} \ee
% Then, the following estimate is standard (see \cite{adn1964}).
% \begin{lemma}  \label{lem-lame}
% For the Lam\'{e}'s equation \eqref{lame1}, one has

% (1) If $f\in W^{k,q}$ for some $q\in(1,\infty),\,\, k\geq0,$ then there exists a unique solution $u\in W^{k+2,q},$ such that
% \begin{equation}
% \displaystyle  \|u\|_{W^{k+2,q}}\leq C(\|f\|_{W^{k,q}}+\|u\|_{L^q});
% \end{equation}
% (2) If $f=\nabla g$ and $g\in W^{k,q}$ for some $q\geq1,\,\,k\geq0,$ then there exists a unique weak solution $u\in W^{k+1,q},$ such that
% \begin{equation}
% \displaystyle  \|u\|_{W^{k+1,q}}\leq C(\|g\|_{W^{k,q}}+\|u\|_{L^q}).
% \end{equation}
% \end{lemma}

The following two lemmas are given in %Theorem 3.2 in \cite{Von1992} and
Propositions 2.6-2.9 in \cite{Aramaki2014}.
\begin{lemma}   \label{lem-vn}
Let $k\geq0$ be a integer, $1<q<+\infty$, and assume that $\Omega$ is a simply connected bounded domain in $\r^3$ with $C^{k+1,1}$ boundary $\partial\Omega$. Then for $v\in W^{k+1,q}$ with $v\cdot n=0$ on $\partial\Omega$, it holds that
\begin{equation}
\displaystyle  \|v\|_{W^{k+1,q}}\leq C(\|\div v\|_{W^{k,q}}+\|\curl v\|_{W^{k,q}}).
\end{equation}
In particular, for $k=0$, we have
\begin{equation}\label{tdu1}
\displaystyle  \|\nabla v\|_{L^q}\leq C(\|\div v\|_{L^q}+\|\curl v\|_{L^q}).
\end{equation}
\end{lemma}
\begin{lemma}   \label{lem-curl}
Let $k\geq0$ be a integer, $1<q<+\infty$. Suppose that $\Omega$ is a bounded domain in $\r^3$ and its $C^{k+1,1}$ boundary $\partial\Omega$ only has a finite number of 2-dimensional connected components. Then for $v\in W^{k+1,q}$ with $v\times n=0$ on $\partial\Omega$, we have
\begin{equation}
\displaystyle \|v\|_{W^{k+1,q}}\leq C(\|\div v\|_{W^{k,q}}+\|\curl v\|_{W^{k,q}}+\|v\|_{L^q}).
\end{equation}
In particular, if  $\Omega$ has no holes, then
\begin{equation}
\displaystyle  \|v\|_{W^{k+1,q}}\leq C(\|\div v\|_{W^{k,q}}+\|\curl v\|_{W^{k,q}}).
\end{equation}
\end{lemma}

Finally, similar to \cite{BKM1984,HLX2012}, we need a Beale-Kato-Majda type inequality with respect to the slip boundary condition \eqref{navier-b} which is given in \cite{cl2019}.
\begin{lemma}\label{lem-bkm}
For $3<q<\infty$, assume that $u\cdot n=0$ and $\curl u\times n=0$ on $\partial\Omega$, $ u\in W^{2,q}$, then there is a constant  $C=C(q,\Omega)$ such that  the following estimate holds
\begin{equation}
\displaystyle \|\na u\|_{L^\infty}\le C\left(\|{\rm div}u\|_{L^\infty}+\|\curl u\|_{L^\infty} \right)\ln(e+\|\na^2u\|_{L^q})+C\|\na u\|_{L^2} +C .
\end{equation}
\end{lemma}

%%      ---------------------------------------------------------------------
%%      ---------------------------ACKNOWLEDGMENTS (OPTIONAL) ---------------
%%      ---------------------------------------------------------------------
\section*{Acknowledgements} This research was partially supported by National Natural Sciences Foundation of China No. 11671027, 11901025, 11971020, 11971217.

%%      For each reference, provide the following information:

% \bibitem{ *** LABEL *** }             %% Give a reference label.
 %*Gray, M., Black, F., and White, A.*              %% Enter author(s) names.
% EXAMPLE:  Gray, M., Black, F., and White, A.

%%      Use the following template for a journal article:
% * Title of article *.                 %% Example: Existence and uniqueness.
% \textit{* Abbreviated journal name *} %% Example: \textit{Comm. Pure Appl. Math.}
% \textbf{* Volume number *}            %% Example: \textbf{72}
% (* Year of publication *),            %% Example: (1993),
% * Issue number [optional],            %% Example: no. 6,
% * Page range *.                       %% Example: 675--690.

%%      Use the following template for a book:
% \textit{* Title of book *}.           %% Example: \textit{Ancient Topology}.
% * Publisher *,                        %% Example: Wiley-Interscience,
% * City of publisher *,                %% Example: New York,
% * Year of publication *.              %% Example: 1993.

%%      ---------------------------------------------------------------------
%%      ------------------------ CONTACT INFORMATION ------------------------
%%      ---------------------------------------------------------------------

%      Place contact information for each author between
%      the \begin{comment} and \end{comment} commands. Include
%      preferred mailing address and e-mail addresses. Please note
%      that these will not print. We will format them for printing
%      during the editing stage.


\begin{thebibliography}{99}

\bibitem{Aramaki2014}
J.~Aramaki.
\newblock {$L^p$} theory for the div-curl system.
\newblock {\em International Journal of Mathematical Analysis}, 8(6):259--271,
  2014.

\bibitem{BKM1984}
J.~T. Beale, T.~Kato, and A.~Majda.
\newblock Remarks on the breakdown of smooth solutions for the {3-D Euler}
  equations.
\newblock {\em Communications in Mathematical Physics}, 94:61--66, 1984.

\bibitem{bl1976}
J.~Bergh and J.~L\"{o}fstr\"{o}m.
\newblock {\em Interpolation Spaces. {An} Introduction}.
\newblock Springer-Verlag, Berlin-Heidelberg-New York, 1976.

\bibitem{BWY2017}
Q.~Bie, Q.~Wang, and Z.-A. Yao.
\newblock Optimal decay rate for the compressible {Navier-Stokes-Poisson}
  system in the critical {$L^p$} framework.
\newblock {\em J. Differential Equations}, 263:8391--8417, 2017.

\bibitem{BB1974}
J.~P. Bourguignon and H.~Brezis.
\newblock Remarks on the euler equations.
\newblock {\em Journal of Functional Analysis}, 15:341--363, 1974.

\bibitem{BZ1997}
M.~Burnat and W.~M. Zajaczkowski.
\newblock On local motion of a compressible barotropic viscous fluid with the
  boundary slip condition.
\newblock {\em Topological Methods in Nonlinear Analysis}, 10(2):195--223,
  1997.

\bibitem{cl2019}
G.C.~Cai and J.~Li.
\newblock Existence and exponential growth of global classical solutions to the
  compressible {Navier-Stokes} equations with slip boundary conditions in {3D}
  bounded domains.
\newblock arXiv:2102.06348.

\bibitem{cf1988}
P.~Constantin and C.~Foias.
\newblock {\em {Navier-Stokes} Equations}.
\newblock Chicago Lectures in Mathematics, University of Chicago Press,
  Chicago, 1988.

\bibitem{Don2003}
D.~Donatelli.
\newblock Local and global existence for the coupled {Navier-Stokes-Poisson}
  problem.
\newblock {\em Quart. Appl. Math.}, 61:345--361, 2003.

\bibitem{FL2018}
Y.~Feng and C.~Liu.
\newblock Stability of steady-state solutions to {Navier-Stokes-Poisson}
  systems.
\newblock {\em J. Math. Anal. Appl.},
  462(2):1679--1694, 2018.

\bibitem{GS2006}
Y.~Guo and W.~Strauss.
\newblock Stability of semiconductor states with insulating and contact
  boundary conditions.
\newblock {\em Arch. Ration. Mech. Anal.}, 179(1):1--30,
  2006.

\bibitem{HL2009}
C.~Hao and H.~Li.
\newblock Global existence for compressible {Navier-Stokes-Poisson} equations
  in three and higher dimensions.
\newblock {\em J. Differential Equations}, 246(12):4791--4812, 2009.

\bibitem{Hoff1995}
D.~Hoff.
\newblock Global solutions of the {Navier-Stokes} equations for
  multidimensional compressible flow with discontinuous initial data.
\newblock {\em J. Differential Equations}, 120:215--254, 1995.

\bibitem{Hoff2005}
D.~Hoff.
\newblock Compressible flow in a half-space with {Navier} boundary conditions.
\newblock {\em J. Math. Fluid Mech.}, 7(3):315--338, 2005.

\bibitem{XJW2003}
L.~Hsiao, Q.~Ju, and S.~Wang.
\newblock The asymptotic behaviour of global smooth solutions to the
  multi-dimensional hydrodynamic model for semiconductors.
\newblock {\em Math. Methods Appl. Sci.},
  26(14):1187--1210, 2003.

\bibitem{XL2010}
L.~Hsiao and H.~Li.
\newblock Compressible {Navier-Stokes-Poisson} equations.
\newblock {\em Acta Math. Sci.}, 30(6):1937--1948, 2010.

\bibitem{HLX2012}
X.~Huang, J.~Li, and Z.~Xin.
\newblock Global well-posedness of classical solutions with large oscillations
  and vacuum to the three-dimensional isentropic compressible {Navier-Stokes}
  equations.
\newblock {\em Comm. Pure Appl. Math.}, 65(4):549--585,
  2012.

\bibitem{KS2008}
T.~Kobayashi and T.~Suzuki.
\newblock Weak solutions to the {Navier-Stokes-Poisson} equation.
\newblock {\em Advances in Mathematical Sciences and Applications},
  18(1):141--168, 2008.

\bibitem{LMZ2010}
H.~Li, A.~Matsumura, and G.~Zhang.
\newblock Optimal decay rate of the compressible {Navier-Stokes-Poisson} system
  in {$\mathbb{R}^3$}.
\newblock {\em Arch. Ration. Mech. Anal.}, 196(2):681--713,
  2010.

\bibitem{LZ2012}
H.~Li and T.~Zhang.
\newblock Large time behavior of solutions to {3D} compressible
  {Navier-Stokes-Poisson} system.
\newblock {\em Sci. China Math.}, 55:159--177, 2012.

\bibitem{lx2016}
J.~Li and Z.~Xin.
\newblock {\em Global Existence of Regular Solutions with Large Oscillations
  and Vacuum}.
\newblock In: Giga Y., Novotny A. (eds) Handbook of Mathematical Analysis in
  Mechanics of Viscous Fluids. Springer, 2016.

\bibitem{LLZ2020}
H.~Liu, T.~Luo, and H.~Zhong.
\newblock Global solutions to compressible {Navier-Stokes-Poisson} and
  {Euler-Poisson} equations of plasma on exterior domains.
\newblock {\em J. Differential Equations}, 269:9936--10001, 2020.

\bibitem{LZ2020}
H.~Liu and H.~Zhong.
\newblock Global solutions to the initial boundary problem of {3-D}
  compressible {Navier-Stokes-Poisson} on bounded domains.
\newblock arXiv:2009.09610.

\bibitem{LXZ2020}
S.~Liu, X.~Xu, and J.~Zhang.
\newblock Global well-posedness of strong solutions with large oscillations and
  vacuum to the compressible {Navier-Stokes-Poisson} equations subject to large
  and non-flat doping profile.
\newblock {\em J. Differential Equations}, 269(10):8468--8508, 2020.

\bibitem{MRS1990}
P.~A. Markowich, C.~A. Ringhofer, and C.~Schmeiser.
\newblock {\em Semiconductor equations}.
\newblock Springer Vienna, 1990.

\bibitem{Nclm1}
C.~Navier.
\newblock M\'{e}moire sur les lois du mouvement des fluides.
\newblock {\em M\'{e}m. Acad. Re. Sci. Inst. Fr.}, 6:389--416, 1823.

\bibitem{Nir1959}
L.~Nirenberg.
\newblock On elliptic partial differential equations.
\newblock {\em Annali della Scuola Normale Superiore di Pisa}, 13(2):115--162,
  1959.

\bibitem{NS2004}
A.~Novotny and I.~Stra\v{s}kraba.
\newblock {\em Introduction to the mathematical theory of compressible flow}.
\newblock Oxford Lecture Ser. Math. Appl., Oxford Univ. Press, Oxford, 2004.

\bibitem{TWW2015}
Z.~Tan, Y.~Wang, and Y.~Wang.
\newblock Stability of steady states of the {Navier-Stokes-Poisson} equations
  with non-flat doping profile.
\newblock {\em SIAM J. Math. Anal.}, 47(1):179--209, 2015.

\bibitem{TW2012}
Z.~Tan and G.~Wu.
\newblock Global existence for the non-isentropic compressible
  {Navier-Stokes-Poisson} system in three and higher dimensions.
\newblock {\em Nonlinear Anal. Real World Appl.}, 13(2):650--664,
  2012.

\bibitem{TZ2013}
Z.~Tan and X.~Zhang.
\newblock Decay of the non-isentropic {Navier-Stokes-Poisson} equations.
\newblock {\em J. Math. Anal. Appl.},
  400:293--303, 2013.

\bibitem{TZ2010}
Z.~Tan and Y.~Zhang.
\newblock Strong solutions of the coupled {Navier-Stokes-Poisson} equations for
  isentropic compressible fluids.
\newblock {\em Acta Math. Sci.}, 30(4):1280--1290, 2010.

\bibitem{WwkW2010}
W.~Wang and Z.~Wu.
\newblock Pointwise estimates of solution for the {Navier-Stokes-Poisson}
  equations in multi-dimensions.
\newblock {\em J. Differential Equations}, 248:1617--1636, 2010.

\bibitem{Wang2012}
Y.~Wang.
\newblock Decay of the {Navier-Stokes-Poisson} equations.
\newblock {\em J. Differential Equations}, 253(1):273--297, 2012.

\bibitem{WW2015}
Y.-Z. Wang and K.~Wang.
\newblock Asymptotic behavior of classical solutions to the compressible
  {Navier-Stokes-Poisson} equations in three and higher dimensions.
\newblock {\em J. Differential Equations}, 259:25--47, 2015.

\bibitem{WW2012}
Z.~Wu and W.~Wang.
\newblock Pointwise estimates of solution for non-isentropic
  {Navier-Stokes-Poisson} equations in multi-dimension.
\newblock {\em Acta Math. Sci.}, 32B(5):1681--1702, 2012.

\bibitem{xx2007}
Y.~Xiao and Z.~Xin.
\newblock On the vanishing viscosity limit for the {3D Navier-Stokes} equations
  with a slip boundary condition.
\newblock {\em Comm. Pure Appl. Math.},
  60(7):1027--1055, 2007.

\bibitem{Zaja1998}
W.~M. Zajaczkowski.
\newblock On nonstationary motion of a compressible barotropic viscous fluid
  with boundary slip condition.
\newblock {\em Journal of Applied Analysis}, 4(2):167--204, 1998.

\bibitem{ZLZ2011}
G.~Zhang, H.~Li, and C.~Zhu.
\newblock Optimal decay rate of the non-isentropic compressible
  {Navier-Stokes-Poisson} system in {$\mathbb{R}^3$}.
\newblock {\em J. Differential Equations}, 250:866--891, 2011.

\bibitem{ZL2012}
Z.~Zhao and Y.~Li.
\newblock Existence and optimal decay rates of the compressible non-isentropic
  {Navier-Stokes-Poisson} models with external forces.
\newblock {\em Nonlinear Anal. Theory Methods Appl.},
  75:6130--6147, 2012.

\bibitem{Zheng2012}
X.~Zheng.
\newblock Global well-posedness for the compressible {Navier-Stokes-Poisson}
  system in the {$L^p$} framework.
\newblock {\em Nonlinear Analysis}, 75:4156--4175, 2012.

\bibitem{zlo2000}
A.~Zlotnik.
\newblock Uniform estimates and stabilization of symmetric solutions of a
  system of quasilinear equations.
\newblock {\em Differential Equations}, 36:701--716, 2000.
\end{thebibliography}
\end{document}